\documentclass[12pt]{article}
\usepackage{graphicx}
\usepackage{amsmath}
\usepackage{amssymb}
\usepackage{theorem}
\usepackage[sans]{dsfont}        
\usepackage[backgroundcolor=blue!20!white, linecolor=blue!20!white,
textsize=footnotesize]{todonotes}
\usepackage[colorlinks]{hyperref}

\numberwithin{equation}{section}

\newtheorem{thm}{Theorem}[section]
\newtheorem{lemma}[thm]{Lemma}

\newtheorem{prop}[thm]{Proposition}

\newtheorem{cor}[thm]{Corollary}
{\theorembodyfont{\rmfamily}

\newtheorem{rmk}[thm]{Remark}
}

\newcommand{\qed}{\hfill \mbox{\raggedright \rule{.07in}{.1in}}}

\newenvironment{proof}{\vspace{1ex}\noindent{\bf
Proof}\hspace{0.5em}}{\hfill\qed\vspace{1ex}}
\newenvironment{pfof}[1]{\vspace{1ex}\noindent{\bf Proof of
#1}\hspace{0.5em}}{\hfill\qed\vspace{1ex}}

\newcommand{\R}{{\mathbb R}}
\newcommand{\C}{{\mathbb C}}
\newcommand{\Z}{{\mathbb Z}}
\newcommand{\D}{{\mathbb D}}

\newcommand{\BV}{\operatorname{BV}}

\newcommand\cB{{\mathcal B}}
\newcommand\cC{{\mathcal C}}
\newcommand\cD{{\mathcal D}}

\newcommand\cN{{\mathcal N}}

\newcommand\cW{{\mathcal W}}
\newcommand\cY{{\mathcal Y}}

\newcommand\bC{{\mathbb C}}
\newcommand\bD{{\mathbb D}}
\newcommand\bE{{\mathbb E}}
\newcommand\bH{{\mathbb H}}
\newcommand\bG{{\mathbb G}}
\newcommand\bN{{\mathbb N}}

\newcommand\bR{{\mathbb R}}

\newcommand\ve{\varepsilon}
\newcommand\eps{\epsilon}
\newcommand\vf{\varphi}

\newcommand\Id{{\mathds{1}}}

\newcommand{\ggen}{\varpi}

\title{Mixing for some non-uniformly hyperbolic systems}

\author{Carlangelo Liverani\footnote{
Dipartimento di Matematica,
Universit\`{a} di Roma (Tor Vergata),
Via della Ricerca Scientifica, 00133 Roma, Italy.
{\tt liverani@mat.uniroma2.it}}\ \ and Dalia Terhesiu\footnote{
Dipartimento di Matematica,
Universit\`{a} di Roma (Tor Vergata),
Via della Ricerca Scientifica, 00133 Roma, Italy.
{\tt daliaterhesiu@gmail.com}}}

\begin{document}

\maketitle

\begin{abstract}
In this work we present an abstract framework that allows to obtain mixing (and in some cases sharp mixing) rates for a reasonable large class of invertible systems preserving an infinite measure.
The examples explicitly considered are the invertible analogue of both Markov and non Markov unit interval maps. For these examples,
in addition to optimal results on mixing and rates of mixing in the infinite case, we obtain results on
the decay of correlation  in the finite case of invertible non Markov maps, which, to our knowledge, were not previously addressed. 

The proposed method consists of a  combination of the framework of operator renewal theory, as introduced in the context of
dynamical systems by Sarig~\cite{Sarig02}, with the framework of function spaces of distributions developed in the recent years along the lines of Blank, Keller and Liverani~\cite{BlankKellerLiverani01}.
\end{abstract}

\section{Introduction}

At present there exist well developed theories that provide subexponential decay of correlation for non-uniformly
expanding maps, culminating with the work of Sarig~\cite{Sarig02}.
For systems with subexponential decay of correlations, previous approaches to~\cite{Sarig02}
for estimating decay of correlations provided only upper bounds. These previous approaches
include
the coupling method of Young~\cite{Young99} (developed upon~\cite{Young98}),
Birkhoff cones techniques adapted to general Young towers by
Maume-Deschamps~\cite{Maume01a} and   the method
of stochastic perturbation developed by Liverani et al~\cite{LiveraniSaussolVaienti99}.
Optimal results on the correlation decay for the class of systems considered in~\cite{LiveraniSaussolVaienti99}
were proved later on by Hu~\cite{Hu04}.

Among other statistical properties,
the method of~\cite{Young99} provides polynomial decay of correlation for non uniformly expanding maps 
that can be modeled by Young towers with polynomially decaying return time tails. The estimates obtained in~\cite{Young99}
were shown to be optimal via the method of  {\em operator
renewal theory} introduced in~\cite{Sarig02} to obtain precise asymptotics and thus, sharp mixing rates.
The latter mentioned method  is an extension of scalar renewal theory from probability theory
to dynamical systems.  Later on, the method of operator renewal theory was substantially
extended and refined by Gou\"ezel~\cite{Gouezel04, Gouezel10}.

In recent work,  Melbourne and Terhesiu~\cite{MT} developed an operator renewal theory framework that recovers the classical notion of mixing
for a  very large class of (non-invertible) dynamical systems with infinite measure.
 We recall that the notion of  ``mixing" for infinite measure preserving systems is very delicate. In general, given a   conservative ergodic infinite
measure preserving transformation $(X,f,\mu)$ with transfer operator $L$, we have $\int L^n v\, d\mu\to 0$,
as $n\to\infty$, for all $v\in L^1(\mu)$. Thus, to recover the classical notion of mixing, one needs to find a sequence $c_n$ and 
a reasonably large class of functions $v$ (within $L^1$) such that $c_n\int L^n v\, d\mu\to C\int v\,d\mu$ for some $C>0$.

In short, the framework of operator renewal theory has been cast (at least implicitly) in a rather general Banach space setting (see, e.g.
\cite{Sarig02, Gouezel04, Gouezel10, MT, Gouezel11}) and has been successfully employed to study the
statistical properties of  both finite and infinite measure preserving, non invertible (eventually expanding) systems. Our aim in this work is to carry out the method of operator renewal theory, in the case of (finite and infinite measure preserving, but focusing on the later) invertible systems. In such a case one would need Banach spaces that allow a direct study of the spectral properties of the transfer operator eliminating altogether, at least in the case of uniformly hyperbolic first return map, the need of coding the system. While until recently it was unclear if such Banach spaces existed at all, the last decade, starting with Blank, Keller and Liverani \cite{BlankKellerLiverani01}, and reaching maturity with \cite{GL06, GL08, BT07, BT2, DemersLiverani08, BaladiGouezel10, DemersZhang11, Li, BuL, BaL, GLP, TsujiiFaure}, has produced an abundance of such spaces. Yet, all such Banach spaces are necessarily Banach spaces of distributions, hence the need to {\em explicitly} cast all the renewal theory arguments in a completely 
abstract form (for example, one must avoid implicit assumptions like the Banach space being a subset of some $L^p$).
 In this work we  provide a set of abstract conditions on dynamical systems (including the non invertible ones)  and develop a corresponding renewal theory framework; this set of hypotheses/conditions includes
the existence of Banach spaces with certain good properties. Moreover, we provide some examples to show that the above mentioned hypotheses are indeed checkable in non trivial cases. Let us explain the situation in more details.

\subsection{Operator renewal theory for invertible systems: the need for new functions spaces}

Given a   conservative (finite or infinite)  measure preserving transformation $(X,f,\mu)$,
renewal theory is an efficient tool for the study of the long term behavior of the transfer operator $L:L^1(X)\to L^1(X)$. 
Fix $Y\subset X$ with $\mu(Y)\in(0,\infty)$.
Let $\varphi:Y\to\Z^+$ be the first return time 
$\varphi(y)=\inf\{n\ge1:f^ny\in Y\}$ (finite almost everywhere by conservativity).  Let $L:L^1(X)\to L^1(X)$ denote the transfer operator 
 for $f$ and define
\begin{equation}\label{eq_TnRn}
T_n=1_YL^n1_Y,\enspace n\ge0, \qquad R_n=1_YL^n1_{\{\varphi=n\}},\enspace n\ge1.
\end{equation}
Thus $T_n$ corresponds to general returns to $Y$ and $R_n$ corresponds to first returns to $Y$. 
The relationship $T_n=\sum_{j=1}^n T_{n-j}R_j$
generalises the notion of scalar renewal sequences (see~\cite{ BGT, Feller} and references therein). The rough idea behind operator renewal theory
is that the asymptotic behavior of the sequence $T_n$ can be obtained via a good understanding of the sequence $R_n$. Apriori assumptions needed to deal with the sequence $R_n$ include  the uniform hyperbolicity of the first return map $F$, along with good spectral properties of the associated transfer operator.

In short, to carry out the method of operator renewal theory for invertible maps $f$,
we need to establish the required spectral results for the transfer operator associated with the uniformly hyperbolic first return map 
$F$. In particular, a spectral gap is needed. Since it is well known that the transfer operators for invertible systems do not have a spectral gap on any of the usual spaces (such as $L^p, W^{p,q}$ or $BV$), unconventional Banach spaces are necessary.

In Section \ref{sec-oprenseq} we will specify exactly which conditions are needed to develop our theory and in the  following sections we obtain several results under such conditions. In section \ref{subsect-classex} we provide examples for which the above conditions are satisfied (which we prove in sections \ref{sec-setup} and \ref{sec-nonMarkov}).

\subsection{Mixing for non--invertible  infinite measure preserving systems}
\label{sec-MISyst}

The techniques in~\cite{MT} are very different from the ones developed for the
framework of operator renewal sequences associated with finite measure~\cite{Sarig02, Gouezel04, Gouezel10}.

In the infinite measure setting a crucial ingredient for the asymptotics of renewal sequences is that $\mu(y\in Y:\varphi(y)>n)=\ell(n)n^{-\beta}$
where $\ell$ is slowly varying\footnote{We recall that a measurable
function $\ell:(0,\infty)\to(0,\infty)$ is slowly varying if $\lim_{x\to\infty}\ell(\lambda x)/\ell(x)=1$ for all $\lambda>0$.
Good examples of slowly varying functions are the asymptotically constant functions and the logarithm.} and
$\beta\in(0,1]$ (see Garsia and Lamperti~\cite{GL} and Erickson~\cite{Erickson} for the setting of scalar renewal sequences). 
Under suitable assumptions on the first return map $T^\varphi$,~\cite{MT} shows that for  a (``sufficiently regular") function $v$  supported on $Y$ 
and a constant  $d_0=\frac{1}{\pi}\sin\beta\pi=[\Gamma(\beta)\Gamma(1-\beta)]^{-1}$, the following hold: i) when $\beta\in(\frac12,1]$
then $\lim_{n\to\infty}\ell(n)n^{1-\beta}T_nv=d_0\int_Y v\,d\mu$, uniformly on $Y$; 
ii) if $\beta\in(0,\frac12]$ and $v\geq 0$ then $\liminf_{n\to\infty}\ell(n)n^{1-\beta}T_nv=d_0\int_Y v\,d\mu$,
pointwise on $Y$ and iii) if $\beta\in(0,\frac12)$
 then $T_nv=O(\ell(n)n^{-\beta})$ uniformly on $Y$.
As shown in~\cite{MT}, the above results on $T_n$ extend to similar results on $L^n$ associated with a large class of non-uniformly expanding systems preserving an infinite measure. 

The results for the case $\beta<1/2$ are  \emph{optimal}  under
the \emph{general assumption} $\mu(\varphi>n)=\ell(n)n^{-\beta}$ (see~\cite{GL}). Under the \emph{additional assumption}
$\mu(\varphi=n)=O(\ell(n)n^{-(\beta+1)})$,  Gou\"ezel~\cite{Gouezel11}
obtains first order asymptotic  for $L^n$ for \emph{all} $\beta\in(0,1)$.

A typical example considered for the study of mixing/mixing rates via renewal operator theory associated with, both,  finite
and infinite measure preserving systems  is the family of Pomeau-Manneville intermittency
maps~\cite{PomeauManneville80}.   To fix notation, we recall the version studied
by Liverani~{\em et al.}~\cite{LiveraniSaussolVaienti99}:
\begin{align} \label{eq-LSV}
f_0(x)=\begin{cases} x(1+2^\alpha x^\alpha), & 0\leq x\leq\frac12 \\ 2x-1, &\frac12<x\leq 1 .
\end{cases}
\end{align}
It is well known that  the statistical properties for $f_0$ can be studied by inducing on  a `good'
set $Y$ inside $(0,1]$, such as the standard set $Y=[1/2,1]$. In particular,  we recall that the inducing method can be used
to show that there exists  a unique (up to scaling) $\sigma$-finite, absolutely continuous invariant measure $\mu_0$:
finite if $\alpha< 1$ and infinite if  $\alpha\geq 1$ ; equivalently, writing  $\beta:=1/\alpha$, $\mu$ is finite if $\beta>1$
and infinite if $\beta\leq 1$.

Let $\varphi_0$ be the return time function to $Y$,  rescale the $f_0$ invariant measure $\mu$
 such that $\mu(Y)=1$ and set $Y_j=\{\varphi_0=j\}$.  We recall that 
$\mu(Y_j)\leq C j^{-(\beta+1)}$ and $|(f_0^j)'(y_j)|^{-1}\leq C j^{-(\beta+1)}$,
for all $y_j\in Y_j$ (see~\cite{LiveraniSaussolVaienti99}). Hence, $\mu(\varphi_0=n)=O(n^{-(\beta+1)})$ and the assumption in
 Gou\"ezel~\cite{Gouezel11} is satisfied, providing first order asymptotic  for $L^n$ for all $\beta\in(0,1)$.

Apart from the above  Markov example, the results in \cite{MT, Gouezel11} apply also to the class
of non-Markovian interval maps,
with indifferent fixed points studied in Zweim\"uller~\cite{Zweimuller98,Zweimuller00}.
For simplicity,  consider the following example that satisfies  the above mentioned additional assumption
in  Gou\"ezel~\cite{Gouezel11}. 

Define a map $f_{0}:[0,1]\to [0,1] $ that on $[0,\frac 12]$ agrees with the map defined by~\eqref{eq-LSV}. On  $(1/2, 1]$,
we assume that there exists a finite partition into open intervals $I_p$, $p\geq 1$ such that  $f_{0}$ is  $C^2$ and strictly monotone in each $I_p$ with
$|f_{0}'|>2$. Moreover, assume that  $f_{0} $ is  topologically
mixing.
Obviously, the  new (not necessarily Markov) map $f_0$ shares many of the properties of the map defined by~\eqref{eq-LSV}. In particular, 
there exists  a unique (up to scaling) $\sigma$-finite, absolutely continuous invariant measure $\mu$:
finite if $\alpha< 1$ and infinite if  $\alpha\geq 1$ ; equivalently, writing  $\beta:=1/\alpha$, $\mu$ is finite if $\beta>1$
and infinite if $\beta\leq 1$. Moreover, given that $Y=[1/2,1]$ , $\varphi_{0}$ is the return time function of
$f_{0}$ to $Y$ and  $Y_j=\{\varphi_0=j\}$, one can easily see that $|f_{0}'(y_j)|^{-1}\leq C j^{-(\beta+1)}$,
for all $y_j\in Y_j$ and that $\mu(\varphi_{0}=n)=O(n^{-(\beta+1)})$.

For more general classes of mixing (in the sense described above) of
non-invertible infinite measure preserving systems (including parabolic maps of the complex plane) we refer to~\cite{MT}. At present it is not entirely clear
how to deal with the infinite measure preserving setting of higher dimensional  non uniformly expanding maps  considered by Hu and Vaienti~\cite{HuVa09}.

\subsection{Mixing rates in the  non invertible case}
\label{sec-mix-rates}

For results on decay of correlation in the finite case of~\eqref{eq-LSV} we refer to~\cite{Sarig02, Gouezel04} and~\cite{Hu04}.
For the infinite case, the method developed in~\cite{MT} yields  mixing rates and \emph{higher} order asymptotics of $L^n$. 
The results in this work suggest that mixing rates in the infinite case can be regarded as the analogue of the decay of correlation in the finite case.

 As shown in~\cite{MT}, 
mixing rates in the infinite measure setting  of $f_0$ can be obtained by exploiting a good enough expansion of 
the tail behavior $\mu_0(\varphi_0>n)$, where $\varphi_0$ is the return time function to a `good'
set $Y$ inside $(0,1]$, such as the standard set $Y=[1/2,1]$. 

Exploiting a modest expansion of 
the tail behavior $\mu_0(\varphi_0>n)$ and good properties of the induced
map $F_0$,
~\cite{MT} shows that for any H{\"o}lder or bounded variation observable
$v:[0,1]\to \R$ with $v$ supported on some compact subset of $(0,1)$, we have
$L^nv=d_0 n^{\beta-1}\int vd\mu_0+O(n^{-(\beta-1/2)})$, uniformly on $Y$. As noted in~\cite{MT},
this rate is optimal for $\beta\geq 3/4$. Exploiting more properties of the return function $\varphi_0$ and of the induced
map $F_0$, improved mixing rates are obtained in~\cite{Terhesiu12}.  The higher order asymptotic of $L^n$ in~\cite{MT, Terhesiu12}
is obtained via the study of associated operator renewal sequences $T_n:\cB\to\cB$, where $\cB$ is the space of  H{\"o}lder or bounded variation
functions.

\subsection{Previous results on mixing/mixing rates for invertible systems}

Adapting Bowen's technique (see \cite{Bow}), Melbourne~\cite{Melbourne_inv} generalizes the results on
mixing  in~\cite{MT}
to infinite measure preserving systems of the form~\eqref{eq-2DLSV} described in section~\ref{subsect-classex}.
The method in~\cite{Melbourne_inv} covers the class of diffeomorphisms that can be modeled by Young
towers, where it is explicitly assumed the quotient of the first return map has 
a Gibbs Markov structure. The results on mixing in~\cite{Melbourne_inv} could, in principle,  be obtained from Theorem~\ref{cor-equiMTl}
and  Corollary~\ref{cor-equilG}. 
However, we limit ourselves to treating explicitly only two classes of examples: i) one covered in ~\cite{Melbourne_inv}
(the example~\eqref{eq-2DnM} described in Section~\ref{subsect-classex}); ii) one not covered in ~\cite{Melbourne_inv} (the example ~\eqref{eq-2DLSV} described in Section~\ref{subsect-classex}).

Under the additional
assumption of exponential contraction along the stable manifold,
~\cite{Melbourne_inv} generalizes the results on
mixing rates in~\cite{MT, Terhesiu12}. As mentioned in~\cite{Melbourne_inv}, without this
further assumption, the employed method does not provide satisfactory results on mixing
rates. On the contrary our Theorem~\ref{thm-mixingrates} below provides optimal mixing rates   for the infinite case
of ~\eqref{eq-2DLSV}, where such uniform contraction along the stable manifold is not required.

Results on (upper bounds for) the decay of correlation in the finite case of~\eqref{eq-2DLSV} can be found in~\cite[Appendix B]{MT12}.

To our knowledge there is no result in the literature that deals with mixing/mixing rates in either the finite or infinite case
for example~\eqref{eq-2DnM}.

\subsection{Main results and outline of the paper}

In Section~\ref{sec-oprenseq}, we describe an abstract framework for operator renewal sequences associated with
non-uniformly hyperbolic systems based on the abstract hypothesis (H1)--(H5), under which results on mixing and mixing rates
hold. Our result on mixing and mixing rates  in this abstract framework are stated and proved in Section~\ref{sec-FO-Tn} (see Theorem~\ref{lemma-FO-Tn-MT}) and Section~\ref{sec-HO-Tn} (see  Theorem~\ref{lemma-HOTn}). These two results establish first and higher order asymptotics of the operator
$T_n$  under the weak assumption (H4)(ii) on the operator $R_n$ (with $T_n,R_n$ as defined in~\eqref{eq_TnRn}). This sort of assumption has not been exploited in previous renewal theory frameworks.
Equally important, the above mentioned abstract results are obtained under a weak assumption (H2), which, we believe, can be verified for interesting hyperbolic transformation (see Remark \ref{rmk:whynot}).

Unfortunately, to state exactly   Theorem~\ref{lemma-FO-Tn-MT} and  Theorem~\ref{lemma-HOTn} requires to establish a bit of notations. Yet, they imply immediately the following easily stated results. Theorem \ref{cor-equiMTl} follows from Theorem \ref{lemma-FO-Tn-MT} , while Theorem \ref{cor-mixingrates} follows from Theorem \ref{lemma-HOTn}.

\begin{thm} \label{cor-equiMTl}
Let $M$ be a manifold and let  $f:M\to M$ be a non-singular transformation w.r.t Lesbeque (Riemmannian) measure $m$. Suppose that there exists 
there exists $Y\subset M$
such that the first return map $F=f^\varphi$  to $Y$ is uniformly hyperbolic (possibly with singularities).

Assume that $F$ satisfies  the functional analytic   assumptions  (H1),(H2) (with $\ggen\beta>\eps_0$), (H3), (H4)(ii) with $\beta \in (1/2,1)$ or (H4)(iii) with $\beta\in(0,1)$, and (H5) stated in Section~\ref{sec-oprenseq}.
For convenience we assume $m(Y)=1$.\footnote{This can always be achieved by rescaling the Riemannian metric}  Let $\mu$ be the invariant measure given by (H1)(iv) and $d_0=[\Gamma(1-\beta)\Gamma(\beta)]^{-1}$. If $v,w:M\to\R$ are $C^\alpha$ (with $\alpha$ as in (H1)(i)) observables supported on $Y$, then
\[
\lim_{n\to\infty}\ell(n)n^{1-\beta}\int_M  v\, w\circ f^n \, d\mu =d_0 \int_M v\, d\mu\int_M w\, d\mu.
\]
\end{thm}

\begin{thm} \label{cor-mixingrates} Assume that $f$ and $F=f^\varphi$ are defined as in Theorem~\ref{cor-equiMTl}.
Assume that $F$ satisfies  the functional analytic   assumptions  (H1),  (H4)(ii) and (H5) stated in Section~\ref{sec-oprenseq}.
 Let $\mu$ be the invariant measure given by (H1)(iv). Let $v,w:M\to\R$ be $C^\alpha$  (with $\alpha$ as in (H1)(i)) observables supported on $Y$. Then the following hold.

\begin{itemize}
\item[i)] Suppose that the strong form of  (H2) holds (that is, $\ggen=1$). Suppose that $\mu(\varphi>n)=c n^{-\beta}+H(n)$
for some $c>0$ and $H(n)=O(n^{-2\beta})$.
Then, there exists constants $d_1,\ldots, d_q$ that depend only on the map $f$
such that for any $\eps_2>\eps_1>\eps_0$ (where $\eps_0$ is defined as in (H4)(ii)):
\begin{equation}\label{eq:infintie-res}
\int_M \hskip-.2cm v\, w\circ f^n \, d\mu =(d_0 n^{\beta-1}+d_1 n^{2\beta-2}+\ldots+d_q n^{(q+1)(\beta-1)}) \int_M \hskip-.2cm v\, d\mu\int_M \hskip-.2cm w\, d\mu
+E_n,
\end{equation}
where $E_n=O(\max\{n^{-(\beta-\eps_2)}, n^{-(1-\beta+\eps_1)/2})$.

\item[(ii)] Assume  the strong form of  (H2) (that is, $\ggen=1$).   Let $\beta>1/2$ and suppose that
  $\mu(\varphi>n)=cn^{-\beta}+b(n)+H(n)$, for some  $c>0$, some function $b$ such that 
$nb(n)$ has bounded variation and $b(n)=O(n^{-2\beta})$,  and some function $H$ such that 
$H(n)=O(n^{-\gamma})$ with $\gamma>2$. 
Then \eqref{eq:infintie-res} holds with the improved rate $E_n=O(n^{-\epsilon_1})$

\item[(iii)]  Assume  the weak form of (H2) (that is, $\ggen\beta>\eps_0$). Suppose that $\mu(\varphi>n)=c n^{-\beta}+H(n)$
for some $c>0$ and $H(n)=O(n^{-2\beta})$. Then \eqref{eq:infintie-res} holds with rate $E_n=O(\max\{n^{-(\ggen\beta-\eps_2)}, n^{-(1-\ggen\beta+\eps_1)/2}\})$.
\end{itemize}
\end{thm}

Given a specific map, the task of checking the hypothesis (H1--H5) is a non trivial one and, alone, can constitute the content of a paper.
Nonetheless, we claim that these hypotheses are reasonable and can be checked in a manifold of relevant examples. To illustrate how to proceed and to convince
the reader that the above claim has some substance, in  Section~\ref{sec-setup} and  Section~\ref{sec-nonMarkov}
 we prove that the abstract hypothesis (H1)--(H5) are satisfied by systems 
of the form~\eqref{eq-2DLSV} and ~\eqref{eq-2DnM} described in Section~\ref{subsect-classex}.  The advantage of focusing on these examples is
that the technicalities are reduced to a bare minimum, which leads to simpler arguments for the verification of (H1--H5).
 We believe that the arguments used in  Section~\ref{sec-setup} and  Section~\ref{sec-nonMarkov} can be followed also
by a reader unfamiliar with the  theory (still in part under construction) of Banach spaces adapted to hyperbolic dynamical systems.
The price to pay for such a choice is that we do not exploit the full force of (H4)(ii) under the weak form of (H2) (i.e. when $\ggen<1$). We believe the latter to be necessary in investigating more complex systems, yet further work is needed to substantiate such a claim (see also Remark \ref{rmk:whynot}).

Apart from the strong property of mixing for invertible infinite measure preserving systems (along with mixing rates),
the present framework allows us to deal with  the property of weak
pointwise dual ergodicity under some weak conditions (under which mixing cannot be proved). For this type of result we refer
to subsection~\ref{rmk-wpde}.
The property of weak p.d.e. has been recently exploited by Aaronson and Zweim\"uller in~\cite{AZ}).
As shown in this work,  weak p.d.e. along with regular variation of the first return time allows one to establish limit theorems  (such
as  Darling Kac)
for infinite measure preserving systems that
are not pointwise dual ergodic (see subsection~\ref{rmk-wpde} for details).

\vspace{-2ex}
\paragraph{Notation}
We use ``big O'' and $\ll$ notation interchangeably, writing
$a_n=O(b_n)$ or $a_n\ll b_n$ as $n\to\infty$ if there is a constant
$C>0$ such that $a_n\le Cb_n$ for all $n\ge1$.

\section{Operator renewal sequences for non-uniformly hyperbolic systems}\label{sec-oprenseq}

In this section we  present an abstract framework that suffices for concrete results on mixing (for maps such as~\eqref{eq-2DLSV} and~\eqref{eq-2DnM}), but general
enough to accommodate a large class of dynamical systems. In particular, it extends the framework of~\cite{Sarig02, Gouezel04} and respectively~\cite{MT}
for operator renewal sequences associated with non-uniformly
expanding maps to the  non-uniformly
hyperbolic context (see the explanatory Remark~\ref{rmk-oprenne}).

Let $M$ be a manifold and  $f:M\to M$ be a
 non-singular transformation w.r.t. Lesbegue (Riemannian) measure $m$. We require that there exists $Y\subset M$
such that the first return map $F=f^\varphi$  to $Y$ is uniformly hyperbolic (possibly with singularities) and satisfies  the functional analytic assumptions listed below.
For convenience we assume $m(Y)=1$, note that this can always be achieved by rescaling the Riemannian metric.

Recall that the transfer operator $R:L^1(m)\to L^1(m)$ for the first return 
map $F:Y\to Y$ is defined by duality on $L^1(m)$ via the formula
$\int_Y Rv\,w\,d m = \int_Y v\,w\circ F\,d m$ for all bounded and measurable $w$. See Remark~\ref{rmk-stdeftrop} for a more explicit description of the transfer operator $R$.
We assume that there exist two Banach spaces of distributions $\cB$, $\cB_w$ supported on $Y$   and some
$\alpha, \gamma>0$ such that

\begin{itemize}\label{H1}
\item[(H1)] \begin{itemize}
\item[(i)] $C^\alpha\subset\cB\subset\cB_w\subset (C^\gamma)'$, where $C^\alpha=C^\alpha (M,\C)$ and
$(C^\gamma)'$ is the dual of $C^\gamma (M,\C)$.\footnote{ We will use systematically a ``prime" to denote the topological dual.}
\item[(ii)] There exists $C>0$ such that for all $h\in\cB$ and $\phi\in C^\alpha$, $h\phi\in\cB$ and 
$\|h\phi\|_{\cB}\leq C\|h\|_{\cB} \|\phi\|_{C^\alpha}$.
\item[(iii)] The transfer operator $R$ associated with $F$ admits a continuous extension to $\cB$, which we still call $R$.
\item[(iv)] The operator $R:\cB\to\cB$ has a simple eigenvalue at $1$
and the rest of the spectrum is contained in a disk of radius less than $1$.
\end{itemize}
\end{itemize}

We note that (H1)(i) should be understood in terms of the usual convention  (see, for instance,
~\cite{GL06, DemersLiverani08}) which we follow thereon:  there exists continuous injective linear maps $\pi_i$ such that $\pi_1(C^\alpha)\subset \cB$, $\pi_2(\cB)\subset \cB_w$ and $\pi_3(\cB_w)\subset (C^\gamma)'$.
We will often leave such maps implicit, unless this might create confusion. In particular, we assume that $\pi=\pi_3\circ\pi_2\circ \pi_1$ is the usual embedding, i.e. for each $ h\in C^\alpha$ and $\phi\in C^\gamma$
\begin{itemize}
\item[(*)]
$\langle \pi(h),\phi\rangle=\int_Y h \phi\, dm$. 
\end{itemize}

Note that, via such identification, the Lebesgue measure $m$ can be identified with the constant function one both in $(C^\gamma)'$ and in $\cB$ (i.e. $\pi(1)=m$). 
Also,  by (H1)(i),  $\cB'\subset (C^\alpha)'$, hence also the dual space can be naturally viewed as a space of distributions. Next, note that $\cB'\supset (C^\gamma)''\supset C^\gamma\ni 1$, thus we have $m\in\cB'$ as well.
Moreover, since $m\in\cB$ and $\langle1,\phi\rangle=\langle\phi,1\rangle=\int \phi\, dm$, $m$ can be viewed as the element $1$  of both spaces
$\cB$ and $(C^\gamma)'$.

By  (H1), the spectral projection $P$ associated with the eigenvalue $1$ is defined by
$P=\lim_{n\to\infty}R^n$. Note that for each $\phi\in C^\alpha$,
\[
\langle P\phi,1\rangle=m(P \phi)=\lim_{n\to\infty}m(1\cdot R^n\phi)=m(\phi)=\langle\phi,1\rangle.
\]
By  (H1)(iv),  there exits a unique $\mu\in\cB$ such that $R\mu=\mu$ and $\langle\mu,1\rangle =1$ . Thus, $P \phi=\mu \langle \phi,1\rangle$.
Also $R'm=m$ where $R'$ is dual operator acting  on $\cB'$. Note that for any $\phi\in C^\gamma$,
\begin{align*}
|\langle \mu, \phi\rangle|=|\langle P1, \phi\rangle|=\left|\lim_{n\to\infty}R^n m(\phi)\right|
= \lim_{n\to\infty}\left|m(\phi\circ F^n)\right|\leq |\phi|_\infty.
\end{align*}
That is,  $|\langle \mu,\phi\rangle)|\leq C|\phi|_\infty$, hence $\mu$ is a measure. Since, for each $\phi\geq 0$, 
\[
\langle P1, \phi\rangle=\lim_{n\to \infty} \langle R1, \phi\rangle=\lim_{n\to \infty} \langle 1, \phi\circ F^n\rangle\geq 0.
\]
It follows that $\mu$ is a probability measure.

Summarizing the above, the eigenfunction associated with the eigenvalue $1$
 is an invariant probability measure for $F$ and  we can write $P1=\mu$.
 
Recall that $\varphi:Y\to \bN$ is the first return time to $Y$. Throughout, we assume that
\begin{itemize}\label{H2}
\item[(H2)]  
there exists $C>0$, $\ggen\in(0,1]$ such that, setting $Y_n=\varphi^{-1}(n)$, for any $n\in\bN$ and $h\in\cB$ we have $\Id_{Y_n} h\in\cB_w$ and
\[
|\langle \Id_{Y_n} h, 1\rangle|\leq C \| h\|_{\cB}\;\mu(Y_n)^\ggen.
\]
\end{itemize}

Given that $\mu$ is the physical probability invariant measure for $F$,  a finite or $\sigma$ finite measure $\mu_0$ for $f$ can be obtained by the
standard push forward method\footnote{For any set $A$ in the $\sigma$-algebra $\mathcal{A}$,
$\mu_0(A)=\sum_{n=0}^\infty\mu(\{\varphi>n\}\cap f^{-n}A)$.} (that goes back to~\cite{Kakutani}). 

\begin{rmk}\label{rmk:whynot} In this paper we will only discuss examples where $\ggen=1$. Nevertheless, in view of \cite{DemersLiverani08}
and \cite{DemersZhang11}, we expect the case $\ggen<1$, together with hypothesis (H4)(ii), to be relevant for more general examples, e.g. symplectic maps with discontinuities (at least in the case when $Y_n$ consists of a uniformly bounded number of connected components).
\end{rmk}

In the \emph{infinite} setting we require that

\begin{itemize}\label{H3}
\item[(H3)]
$\mu(y\in Y:\varphi(y)>n)=\ell(n)n^{-\beta}$ where $\ell$ is slowly varying and
$\beta\in[0,1]$.
\end{itemize}

\begin{rmk}\label{rmk-decay}
In this paper, we do \emph{not} treat explicitly the \emph{finite} measure case, that is $\mu(y\in Y:\varphi(y)>n)=O(n^{-\beta}),\beta>1$, but just notice
that~\cite[Theorem 1]{Gouezel04,Sarig02} formulated in terms of abstract Banach spaces holds in this framework under (H4)(iii) below: see Subsection~\ref{sec-decay}.
\end{rmk}

Let $\D=\{z\in\C:|z|<1\}$ and
$\bar\D=\{z\in\C:|z|\le1\}$.
Given $z\in\bar\D$, we define the perturbed transfer operator $R(z)$ (acting on $\cB, \cB_w$)  by
$R(z)v=R(z^\varphi v)$.   Also, for each $n\ge1$, we define
$R_n$ (acting on $\cB, \cB_w$) by  $R_nv=R(1_{\{\varphi=n\}}v)$. We assume that at least one version of (H4) below holds:
\begin{itemize}\label{H}
\item[(H4)] 
\begin{itemize}
\item[(i)] $R_n:\cB\to\cB$ are bounded operators satisfying
$\sum_{n=1}^\infty\|R_n\|_{\cB\to\cB}<\infty$.  
\item[(ii)] $\|R_n\|_{\cB\to\cB_w}\ll c_n$,
where $\sum_{j>n}c_n\ll n^{-(\beta-\epsilon_0)}$ with $\beta\in (1/2,1)$ and $\epsilon_0<\max\{2\beta-1,1-\beta\}$.  
\item[(iii)]  $\|R_n\|_{\cB}\ll n^{-(\beta+1)}$.
\end{itemize}
\end{itemize}

\begin{rmk}Assumption (H4)(i) is not sufficient for obtaining mixing rates in neither the finite nor the infinite measure case.
In this paper we use (H4)(i) along with the present abstract set up to establish weak pointwise dual ergodicity
for infinite measure preserving non uniformly hyperbolic systems  (see Subsection~\ref{rmk-wpde}).

Assumptions of the type (H4)(ii) have not been used in  previous renewal theory abstract frameworks. In this
work we use this to obtain mixing rates for infinite measure preserving systems satisfying (H3) (see Sections~\ref{sec-FO-Tn}
and~\ref{sec-HO-Tn}).
 We believe that a version of ~\cite[Theorem 1]{Gouezel04,Sarig02}
can be proved under the weaker condition (H4)(ii) above (formulated as appropriate for the finite case). However, the arguments are very different from the ones used in this paper to deal with the infinite measure
case. We postpone
this problem to  a future note.

Assumption (H4)(iii) is standard in previous renewal theory frameworks (see Remark~\ref{rmk-oprenne}).
\end{rmk}

Assumption (H4)(i) (or (H4)(ii)) ensures that $R(z)=\sum_{n=1}^\infty R_nz^n$ is a well defined family of operators from
$\cB$ to $\cB$ (or from $\cB$ to $\cB_w$). 
Also, we notice that  (H1) and (H4)(i) (or (H4)(ii)) ensure that $z\mapsto R(z)$,  $z\in\bar\D$,  is a continuous family of bounded operators (analytic on $\D$) from $\cB$ to $\cB$ (or from $\cB$ to $\cB_w$).
  Throughout we assume:

\begin{itemize}\label{H5}
\item[(H5)]
\begin{itemize}
\item[i)] There exist $C>0$ and $\lambda>1$ such that for all $z\in\bar\D$ and for all
$h\in\cB$, $n\geq 0$,
\[
\|R(z)^n h\|_{\cB_w}\le C|z|^n\|h\|_{\cB_w},\quad \|R(z)^n 
h\|_{\cB}\le C\lambda^{-n}|z|^n\|h\|_{\cB}
+C|z|^n\|h\|_{\cB_w}
\]
\item[ii)] For $z\in\bar\D\setminus\{1\}$, the spectrum of $R(z):\cB\to\cB$ does
not contain $1$.
\end{itemize}
\end{itemize}

In particular,  we note that (H1), (H4)(i) (or (H4)(ii)) and (H5)(i) implies that for $z\in\D$,
$z\mapsto (I-R(z))^{-1}$ is an analytic family of bounded linear operators from $\cB$ to $\cB$ (or from $\cB$ to $\cB_w$).\footnote{ Indeed note, by (H5)(i), the spectral radii of $R(z)$, $z\in\D$, are strictly smaller than one and  that $(I-R(z_0))^{-1}-(I-R(z_1))^{-1}=(z_0-z_1)\sum_{n=1}^{\infty}\sum_{k=0}^nz_0^kz_1^{n-k-1} (I-R(z))^{-1} R_n(I-R(z))^{-1}$ which implies complex differentiability in $\D$ with respect to the relevant topologies.}
Define $T_n:\cB\to\cB$ for $n\ge0$ and $T(z):\cB\to\cB$
for $z\in\bar\D$ by setting $T_0=I$ and
\[
T_nv=\sum_{k=1}^\infty\sum_{i_1+\ldots + i_k=n}R_{i_1}\ldots R_{i_k} v, n\geq 1\;;
\qquad T(z)=\sum_{n=0}^\infty T_nz^n.
\]

By a standard computation we have that $T(z)=I+R(z)T(z)$ for all $z\in\bD$. Since by (H5)(i), the spectrum of $R(z)$ does not contain 1, for all $z\in\D$,
 we have  the renewal equation 
\begin{equation*}\label{eq:renewal}
T(z)=(I-R(z))^{-1}, z\in\D.
\end{equation*}
 By (H5)(ii), $T(z)$ extends continuously to  $\bar\D\setminus \{1\}$.
Moreover, on $\D$, $z\to T(z)=\sum_{n=0}^\infty T_nz^n$ is an analytic  family of bounded linear operators
 from $\cB$ to $\cB$ (or from $\cB$ to $\cB_w$).

\begin{rmk}\label{rmk-Ident}
We notice that if $L :L^1(m)\to L^1(m)$  is the transfer operator of the original transformation $f:M\to M$, then the sequences of operators $R_n, T_n$ defined
in this section coincide with the sequences of operators defined in~\eqref{eq_TnRn}. For $R_n$ this is simply the bare definition, while for $T_n$,
it follows by decomposing the itinerary of $f:Y\to Y$ into
consecutive returns to $Y$ (see, for instance,~\cite{Gouezel04}).
\end{rmk}

\begin{rmk}\label{rmk-oprenne}In  the context of non-uniformly expanding maps preserving a \emph{finite} invariant measure
$\mu$, the functional analytic
assumption on $F$  summarizes as follows. It is assumed that there exists a Banach space
$\cB$ (for non-uniformly expanding interval maps $\cB$ is H{\"o}lder or
BV)  such that H1(ii) and  (H5)(ii) hold for $R(1)$ and
$R(z)$, respectively, as operators on $\cB$.
Moreover, one requires that (H4)(ii)  holds under the strong norm $\|.\|$  on
$\cB$ for some $\beta>1$ (see~\cite{Sarig02,  Gouezel04}).
We also refer to~\cite{MT12}, where (H4)(ii) reduces to  $\sum_{n=1}^\infty\sum_{j>n}
\|R_j\|<\infty$. In the case  of non-uniformly expanding
maps preserving an \emph{ infinite} invariant measure $\mu$, the assumption (H3) is
crucial (see~\cite{MT}).
\end{rmk}

\begin{rmk}\label{rmk-stdeftrop}
Note that, using  convention (*), one has the following. Identifying a measure
$h$ that is absolutely continuous w.r.t
 $m$  with its density (which will be again called $h$), the space of measures absolutely continuous
 w.r.t. $m$ can be canonically identified with $L^1(Y,\R,m)$.  Restricting to $L^1(Y)\subset
(C^\gamma)'$ and writing $DF:=|det F|$, we have $R h= h\circ F^{-1} |DF\circ F^{-1}|^{-1}$. Thus, our operator $R$ on $\cB$ is a extension of the usual transfer operator.
\end{rmk}

\begin{rmk}
 Recall that a measure $\nu$ is {\em physical} if there exists a 
measurable set $A$, $m(A)>0$, such that, for each continuous function $\phi$, $\lim\limits_{n\to \infty}\frac 1n\sum_{k=0}^{n-1}\phi\circ F^k(x)=\nu(\phi)$,
for each $x\in A$.
In the present case, by hypotheses (H1)(iii)-(iv), we have that $\lim\limits_{n\to \infty}\frac 1n\sum_{k=0}^{n-1}\phi\circ F^k(x)=\mu(\phi)$ for $m$-almost all $x$. It follows that $\mu$ is the unique physical measure of $F$. 
Indeed, suppose there exists $B\subset Y$ for which the last limit is larger than $\mu(\phi)+\ve$, for some $\ve>0$ 
(the case of the limit being smaller being treated similarly). Then, by Lusin theorem with respect to $m$ there exists a
$C^1$ function $h_\ve$ such that $\|\Id_B-h_\ve\|_{L^1(m)}<\frac{\ve m(B)}{4|\phi|_\infty}$. But then
\begin{align*}
\mu(\phi)m(h_\ve)&=
\lim_{n\to\infty}\frac{1}{n}\sum_{k=0}^{n-1} m(h_\ve \phi\circ F^k)\geq\lim_{n\to\infty}\frac{1}{n}
\sum_{k=0}^{n-1} m(\Id_B \phi\circ F^k)-\frac {\ve m(B)} 4\\
&\geq m(B)\mu(\phi)+\frac {3\ve m(B)}4\geq m(h_\ve)\mu(\phi)+\frac {\ve m(B)}2
\end{align*}
which is a contradiction.
\end{rmk}

\section{Asymptotics of $T(z)$}
\label{sec-FO_Tz}

\subsection{Asymptotic results under (H1), (H2) with $\ggen=1$, (H3),  (H4)(i) and  (H5)(i)}
\label{sec-FO1_Tz}

Our aim in this section is to estimate $\|T(z)\|_{\cB}$, $z\in\D$  as $z\to 1$ under the weak hypothesis (H4)(i) and
(H5)(i). Recall that (H5)(i) implies that the spectrum of $R(z)$ does not contain 1, for all $z\in\D$.

As in the framework of~\cite{MT, MT11}, the asymptotics of $T(z)$, $z\in\bar{\D}$ depends essentially on the asymptotics of
the eigenvalue $\lambda(z)$ of $R(z)$ defined in a neighborhood of $1$. 
(and Lemma \ref{lemma-FO-Tn-G})

Hypothesis (H1)(iv), (H4)(i) and (H5)(i) plus standard perturbation theory imply that, for $z$ in a neighbourhood of one, $R(z)$
has a simple maximal eigenvalue, hence a spectral decomposition of the type $R(z)=\lambda(z) P(z)+\widetilde Q(z)$
where $\| \widetilde Q(z)^n\|_{\cB}\leq C\sigma^n$ for all $n\in\bN$ and some fixed $C>0$ and $\sigma<1$.
Moreover, $P(z)$ is a family of rank one projectors which, by (H4)(i) depend continuously on $z$. We can then write $P(z)=m(z)\otimes v(z)$, with $m(z)(v(z))=1$, where
$m(1)$ is the Lebesgue measure and $v(1)=\mu$ is the invariant probability measure. For $z$ close enough to $1$
we can normalise $v(z)$ so that $\langle v(z),1\rangle=1$, hence
\begin{equation}\label{eq:step1}
\lambda(z)=\langle R(z)v(z), 1\rangle.
\end{equation}

For $z\in \bar\D\cap B_\delta(1)$, $R(z)=\lambda(z) P(z)+ \widetilde Q(z)=\lambda(z)P(z)+ R(z)Q(z)$, where $Q(z)=I-P(z)$. Hence, for $z\in \bar\D\cap B_\delta(1)$, $z\neq 1$,
\begin{equation}\label{eq-T(z)}
T(z)=(1-\lambda(z))^{-1}P+(1-\lambda(z))^{-1}(P(z)-P)+(I-R(z))^{-1}Q(z).
\end{equation}

By (H1)(iv) and (H5)(i), there exists $\delta,C>0$ such that
$\|(I-R(z))^{-1}Q(z)\|_{\cB}\le C$ for $z\in \bar\D\cap
B_\delta(1)$, $z\neq 1$.
By~\eqref{eq-T(z)}, it remains to obtain the asymptotics of $(1-\lambda(z))^{-1}$ and $(1-\lambda(z))^{-1}(P(z)-P)$.
First, we note that (H4)(i), together with standard perturbation theory, implies that as $z\to 1$,
\[
\|R(z)-R\|_{\cB}\to 0,\quad \|P(z)-P\|_{\cB}\to 0,\quad |\lambda(z)-1|\to 0.
\]

The next result provides more precise information on the asymptotics of $\lambda(z)-1$, generalizing
the known result of~\cite{AaronsonDenker01} (see also ~\cite[Lemma A.4]{MT11}) to the present 
abstract framework.

\begin{lemma}\label{lemma-asymplambda0}Assume (H1), (H2) with $\ggen=1$, (H3), (H4)(i) and (H5)(i). Then as $u,\theta\to 0$,
\[
1-\lambda(e^{-(u+i\theta)}) =\Gamma(1-\beta)\ell(1/|u-i\theta|)(u-i\theta)^\beta(1+o(1)).
\]
\end{lemma}

\begin{proof}
By~\eqref{eq:step1}, $\lambda(z) v(z)=R(z)v(z)$, with $\lambda(1)=1$ and $v(1)=\mu$. Recalling that $Y_n=\{\vf=n\}$, we can write
\[
\begin{split}
\lambda(z)&=\langle R(z)v(z),1\rangle=\sum_nz^n\langle R\left[\Id_{Y_n}v(z)\right],1\rangle=\sum_nz^n\langle \Id_{Y_n}v(z),1\rangle\\
&=\sum_nz^n\langle \Id_{Y_n} [v(z)-v(1)],1\rangle+\mu(z^\vf),
\end{split}
\]
which yields
\begin{equation}\label{eq-ev-Gou}
1-\lambda(z)=\mu(1-z^\vf)+\sum_n(z^n-1)\langle \Id_{Y_n} [v(1)-v(z)],1\rangle.
\end{equation}

As shown in~\cite[Proof of Lemma A.4]{MT11}, under (H3), the precise asymptotic of $\mu(1-z^\vf)$ on $\bar\D$ , a generalization of the more standard result for the precise asymptotic of $\Psi(z)$ on the unit circle (see for instance~\cite{ GL}), 
is given by
\begin{equation}\label{eq-GL} 
\mu(1-e^{-(u+i\theta)\varphi})=\Gamma(1-\beta)\ell(1/|u-i\theta|)(u-i\theta)^\beta(1+o(1)),\mbox{ as }u,\theta\to 0.
\end{equation}
On the other hand, by (H2) with $\ggen=1$, and Lemma~\ref{lem:compone},
\[
\begin{split}
\left|\sum_n(z^n-1)\langle \Id_{Y_n} [v(1)-v(z)],1\rangle\right|&\leq C\sum_n|z^n-1|\mu(\Id_{Y_n}) \|v(1)-v(z)\|_{\cB}\\
&=o(\mu(|1-z^\vf|)).
\end{split}
\]
\end{proof}

\begin{cor}\label{eq-asymp0}Assume the setting of Lemma~\ref{lemma-asymplambda0}. Then
\begin{equation*}
T(e^{-u+i\theta})=\Gamma(1-\beta)^{-1}\ell(1/|u-i\theta|)^{-1}(u-i\theta)^{-\beta} P+E,\mbox{ as } u,\theta\to 0,
\end{equation*}
 where  $\|E\|_{\cB}=o(\ell(1/|u-i\theta|)^{-1}|u-i\theta|^{-\beta})$.
\end{cor}
\begin{proof} Recall $\|P(z)-P\|_{\cB}\to 0$. Also, by (H1)(iv) and (H5)(i), there exists $\delta,C>0$ such that
$\|(I-R(z))^{-1}Q(z)\|_{\cB}\le C$ for $z\in \bar\D\cap
B_\delta(1)$, $z\neq 1$.

The conclusion follows from this together with~\eqref{eq-T(z)} and  Lemma~\ref{lemma-asymplambda0}.
\end{proof}

The result below was used in the proof of  Lemma~\ref{lemma-asymplambda0} and will be further used in the next sections.

\begin{lemma}\label{lem:compone} Assume (H3). For each $\ve>0$, there exists $C>0$ such that, for all $\ggen\in (0,1]$, for all $u\geq 0$ and all $\theta\in (-\pi,\pi]$, we have
\[
\sum_k|1-e^{-k(u-i\theta)}|\mu(Y_k)^\ggen\leq C|u-i\theta|^{\ggen\beta-(1-\ggen)\ve}\ell(|u-i\theta|^{-1}).
\]
Moreover, for all $h\leq\min\{u,|\theta|\}$,
\[
\sum_k|e^{-k(u-i(\theta+h))}-e^{-k(u-i\theta)}|\mu(Y_k)^\ggen\leq C h^{\ggen\beta-(1-\ggen)\ve}\ell(1/h).
\]
\end{lemma}

\begin{proof}Let $G(x)=\mu(\vf\geq x)$.
By (H3), we have, for each $\ve>0$,\footnote{ By $\lfloor x\rfloor$ we designate the integer part of $x$.}
\[
\begin{split}
\sum_k&|1-e^{-k(u-i\theta)}|\mu(Y_k)^\ggen\leq \sum_k|1-e^{-k(u-i\theta)}|\{G(k)-G(k+1)\}^\ggen\\
&\leq C_0\sum_k\min\{1,k|u-i\theta|\}G(k)^{\ggen-1}[G(k)-G(k+1)]\\
&\leq C_0\lim_{L\to\infty}\int_0^L \min\{1,\lfloor x\rfloor|u-i\theta|\}\lfloor x\rfloor^{(1-\ggen)(\beta+\ve)}dG(x)\\
&\leq C_1 |u-i\theta|+C_1\lim_{L\to\infty}\int_0^L \min\{1,x |u-i\theta|\} x^{(1-\ggen)(\beta+\ve)}dG(x)\\
&=C_1|u-i\theta|\int_0^{|u-i\theta|^{-1}}x^{-\ggen\beta+(1-\ggen)\ve}\ell(x) dx+C_1 |u-i\theta|\\
&\quad+C_2\lim_{L\to\infty} \int_{|u-i\theta|^{-1}}^L x^{-\ggen\beta-1+(1-\ggen)\ve}\ell(x) dx+L^{(1-\ggen)(\beta+\ve)}G(L)\\
&\leq C_1|u-i\theta|^{\ggen\beta-(1-\ggen)\ve}\ell(|u-i\theta|^{-1})\int_0^{1}x^{-\ggen\beta+(1-\ggen)\ve}\frac{\ell(x|u-i\theta|^{-1})}{\ell(|u-i\theta|^{-1})} dx\\
&\quad+C_2|u-i\theta|^{\ggen\beta-(1-\ggen)\ve}\ell(|u-i\theta|^{-1})\int_{1}^\infty x^{-\ggen\beta-1+(1-\ggen)\ve}\frac{\ell(x|u-i\theta|^{-1})}{\ell(|u-i\theta|^{-1})} dx\\
&\quad +C_1 |u-i\theta|.
\end{split}
\]
Note that, by Potter bound \cite[Theorem 1.5.6]{BGT}, the integrand in both integral are uniformly integrable, hence, by Lebesgue Dominate convergence Theorem,
\[
\sum_k |1-e^{-k(u-i\theta)}|\mu(Y_k)^\ggen\leq C|u-i\theta|^{\ggen\beta-(1-\ggen)\ve}\ell(|u-i\theta|^{-1}).
\]
The second part of the conclusion follows by a similar calculation.
\end{proof}

\subsection{Asymptotic under (H1), (H2), (H3), (H5) and (H4)(ii)}

In this section, we obtain the asymptotic, along with explicit bounds
on the continuity, for $\|T(z)\|_{\cB\to \cB_w}$, for  $z$ in a neighborhood of $1$ (see Corollary~\ref{cor-estTdifKL} and Proposition~\ref{prop-estTdifKL}) under (H2) with $\ggen\leq 1$ and (H4)(ii).
To complete the picture of the asymptotic of $T(z)$,  $z\in\D$ we then estimate the derivative of $T(z)$ for $z$ outside
a neighborhood of $1$ (see Corollary~\ref{cor-deriv}); for this estimate the full force of (H5)(ii) is required.

Reasoning as in~\cite[Remark 4]{KellerLiverani99}, we note that  (H1)(iv), (H4)(ii) and (H5)(i) imply that, for $z$ in a neighborhood of $1$, the spectrum of $R(z)$ has a maximal simple isolated eigenvalue $\lambda(z)$
with $\lambda(1)=1$. We have the following analogue of Lemma~\ref{lemma-asymplambda0}.

\begin{lemma}\label{lemma-asymplambda2}Assume (H1), (H3), (H4)(ii) and (H5)(i). Assume (H2) with $\ggen\beta>\epsilon_0$. Then  for any 
and $\eps_1\in (\eps_0,1)$ the following holds as $u,\theta\to 0$:
\[
1-\lambda(e^{-(u+i\theta)}) =\Gamma(1-\beta)\ell(1/|u-i\theta|)(u-i\theta)^\beta +O(|u-i\theta|^{\beta+\ggen\beta-\eps_1}\ell(1/|u-i\theta|)).
\]
\end{lemma}

Reasoning as in the proof of  Lemma~\ref{lemma-asymplambda0}, the conclusion of  Lemma~\ref{lemma-asymplambda2} will follow once we estimate the second term
of~\eqref{eq-ev-Gou}. Since this estimate requires dealing with $\|v(z)-v(1)\|_{\cB\to\cB_w}$ we need the following result from~\cite{KellerLiverani99}.

\begin{lemma}\label{cor-estKL}Assume  (H1), (H4)(ii) and (H5)(i).
Then there exists $\delta_0>0$  such that the following holds  for each $\eta\in(0,1)$
for all $e^{-u+i\theta}\in B_{\delta_0}(1)$ , for all $h\leq \min\{|\theta|,u\}$ and for some $C_\eta>0$
\begin{align*}
\|P(e^{-u+i\theta})-P\|_{\cB\to\cB_w}
\le C_\eta |u-i\theta|^{\eta(\beta-\epsilon_0)},\quad
\|P(e^{-u+i\theta}))-P( e^{-u+i(\theta-h)})\|_{\cB\to\cB_w}
\le C_\eta h^{\eta(\beta-\epsilon_0)}.
\end{align*}
Moreover, the same estimates hold for the families $Q(z)$ and $v(z)$.
\end{lemma}

\begin{proof}
By (H4)(ii) and arguing as in the proof of Lemma \ref{lem:compone}, we have
\begin{equation}\label{eq-contR}
\begin{split}
\|R(e^{-u+i\theta})&-R(1)\|_{\cB\to\cB_w}\leq \sum_k |1-e^{-k(u-i\theta)}|\|R_k\|_{\cB\to\cB_w}\\
&\leq C\int_0^{|u-i\theta|^{-1}} x^{-(\beta-\epsilon_0)} dx+C|u-i\theta|^{\beta-\epsilon_0}
\leq C_1|u-i\theta|^{\beta-\epsilon_0}.
\end{split}
\end{equation}
We can then apply~\cite[Corollary 1]{KellerLiverani99}. Indeed, (H5)(i) correspond to hypotheses (2,3) in \cite{KellerLiverani99}. The above estimate yields hypothesis (5) in \cite{KellerLiverani99}, while hypothesis (4) in \cite{KellerLiverani99}
is redundant in the present case as explained by ~\cite[Remark 6]{KellerLiverani99}.  

By~\cite[Corollary 1]{KellerLiverani99},  there exists $\delta_0>0$  such that for each $\eta\in(0,1)$ and for
for all $e^{-u+i\theta}\in B_{\delta_0}(1)$,\footnote{ In fact, working a bit more (see  \cite[Remark 5]{KellerLiverani99}) one has
\[
\|P(e^{-u+i\theta})-P\|_{\cB\to\cB_w}\leq C|u-i\theta|^{\beta-\epsilon_0}\ln|u-i\theta|^{-1}.
\]
}
\[
\|P(e^{-u+i\theta})-P\|_{\cB\to\cB_w}\leq C_\eta|u-i\theta|^{\eta(\beta-\epsilon_0)}.
\]
The estimate for $\|P(e^{-u+i\theta}))-P( e^{-u+i(\theta-h)})\|_{\cB\to\cB_w}$ follows similarly. The estimates for the family $Q(z)$ are an immediate consequence.
For completeness, we provide the argument for estimating  $\|v(e^{-u+i\theta})-v(1)\|_{\cB\to\cB_w}$. The estimate for $\|v(e^{-u+i\theta}))-v( e^{-u+i(\theta-h)})\|_{\cB\to\cB_w}$ follows similarly. Note that
\[
| m(e^{-u+i\theta})(1)-m(1)(1)|=|\langle P(e^{-u+i\theta})1-P(1)1,1\rangle|\leq C_\eta|u-i\theta|^{\eta(\beta-\epsilon_0)}.
\]
Thus,
\[
\begin{split}
\|v(e^{-u+i\theta})-v(1)\|_{\cB_w}&=\left\|\frac{1}{m(e^{-u+i\theta})(1)}P(e^{-u+i\theta})1-\frac{1}{m(1)(1)}P(1)1\right\|_{\cB_w}\\
&\leq C_\eta|u-i\theta|^{\eta(\beta-\epsilon_0)}.
\end{split}
\]
\end{proof}

We can now complete

\begin{pfof}{ Lemma~\ref{lemma-asymplambda2}}By equation~\eqref{eq-ev-Gou},
\[
 1-\lambda(e^{-u+i\theta})=
\sum_nz^n\langle \Id_{Y_n} [v(e^{-u+i\theta})-v(1)],1\rangle+\mu(1-e^{(-u+i\theta)\vf}).
\]
 By Lemma~\ref{cor-estKL}, for all $u,|\theta|<\delta_0$
we have $\|v(e^{-u+i\theta})-v(1)\|_{\cB\to\cB_w}
 \leq C_\eta|u-i\theta|^{\eta(\beta-\epsilon_0)}$ for any $\eta\in (0,1)$.
By (H2),  $\sum_nz^n\langle \Id_{Y_n} [v(e^{-u+i\theta})-v(1)],1\rangle\leq C \|v(e^{-u+i\theta})-v(1)\|_{\cB\to\cB_w}\sum_k|1-e^{-k(u-i\theta)}|\mu(Y_k)^\ggen$.
By assumption,  $\ggen\beta>\eps_0$. By  Lemma~\ref{lem:compone}, $\sum_k|1-e^{-k(u-i\theta)}|\mu(Y_k)^\ggen\leq C|u-i\theta|^{\ggen\beta-(1-\ggen)\ve}\ell(|u-i\theta|^{-1})$.

Putting these together, we obtain that as $u,\theta\to 0$,
$1-\lambda(e^{-u+i\theta})=d_0\ell(1/|u-i\theta|)(u-i\theta)^{\beta}+O(|u-i\theta|^{\eta(\beta-\epsilon_0)+\ggen\beta-(1-\ggen)\ve}\ell(1/|u-i\theta|))$.
The conclusion follows by taking $\eta$ close to $1$, $\ve$ small enough and recalling $\eps_1>\eps_0$.
\end{pfof}

  The result below below is
 an analogue of  Corollary~\ref{eq-asymp0} under (H4)(ii).

  \begin{cor}\label{cor-estTdifKL} Assume the setting of  Lemma~\ref{lemma-asymplambda2}. Choose $\delta_0>0$ such that  Lemma~\ref{cor-estKL} holds.
 Then, for all $u,|\theta|<\delta_0$, for any $\ggen$ such that $\ggen\beta>\eps_0$ and for any $\eps_1>\eps_0$,
\begin{equation*}
T(e^{-u+i\theta})=\Gamma(1-\beta)^{-1}\ell(1/|u-i\theta|)^{-1}(u-i\theta)^{-\beta} P+E,
\end{equation*}
 where $E$ is an bounded operator in $\cB$ such that $\|E\|_{\cB\to\cB_w}=O(|u-i\theta|^{-(1-\ggen)\beta-\epsilon_1})$.
 \end{cor}
 
 \begin{proof}Recall~\eqref{eq-T(z)}.
By Lemma~\ref{cor-estKL}, $\|P(e^{-u+i\theta})-P\|_{\cB\to\cB_w}=O(|u-i\theta|^{\beta-\epsilon_1})$, for any $\eps_1>\eps_0$. Also, by (H1)(iv) and (H5)(i), there exists $C>0$ such that
$\|\left(I-R(z))Q(z)\right)^{-1}\|_{\cB}\le C$ for $z\in \bar\D\cap
B_\delta(1)$, where $\delta>0$ is such that $\lambda(z)$ is simple for $z\in \bar\D\cap
B_\delta(1)$. The conclusion follows from these estimates together with~\eqref{eq-T(z)} and  Lemma~\ref{lemma-asymplambda2}.~\end{proof}

\begin{rmk}\label{rmk-normw} From the above short  proof, we also have that under  (H1), (H2), (H3), (H4)(ii) and (H5)(i),
$\|T(e^{-u+i\theta})\|_{\cB}\ll \ell(1/|u-i\theta|)^{-1}(u-i\theta)^{-\beta}$, for all $u,|\theta|<\delta$,
 where $\delta>0$ is such that $\lambda(e^{-u+i\theta})$ is well defined for $e^{-u+i\theta}\in \bar\D\cap
B_\delta(1)$.
\end{rmk}

 The next result  provides explicit bounds
on the continuity of $T(z)$, for $z$ a neighborhood of $1$.

\begin{prop}\label{prop-estTdifKL} Assume (H1), (H3),  (H4)(ii) and (H5)(i). Assume (H2) with $\ggen\beta>\eps_0$.
Choose $\delta_0>0$ such that  Lemma~\ref{cor-estKL} holds.
Then the following hold for any $\eps_1>\eps_0$, for all $u,|\theta|<\delta_0$ and for all $h\leq \min\{|\theta|,u\}$.

\begin{itemize}
\item[i)]$|\lambda(e^{-u+i\theta})-\lambda(e^{-u+i(\theta-h)})|\ll
h^\beta\ell(1/h)+h^{\ggen\beta-\epsilon_1}|u-i\theta|^{\beta}\ell(1/|u-i\theta)$.

\item[ii)]Also,
\begin{align*}
\|T(e^{-u+i\theta})-T(e^{-u+i(\theta-h)})\|_{\cB\to\cB_w}&\ll
\ell(1/|u-i\theta|)^{-2}\ell(1/h)h^{\beta}|u-i\theta|^{-2\beta}\\
&+\ell(1/|u-i\theta|)^{-1}h^{\ggen\beta-\epsilon_1}|u-i\theta|^{-\beta}.
\end{align*}
\end{itemize}
\end{prop}

\begin{proof} [i)]
Put $\Delta_\lambda=\lambda(e^{-u+i\theta})-\lambda(e^{-u+i(\theta-h)})$.
By~\eqref{eq-ev-Gou},
\begin{align*}
\Delta_\lambda& =\mu(e^{(-u+i\theta)\vf}-e^{(-u+i(\theta-h))\vf})+\sum_n (e^{(-u+i\theta)n}-1)\langle \Id_{Y_n}[v(e^{-u+i\theta})-v(e^{-u+i(\theta-h)})],1\rangle\\
&+\sum_n (e^{(-u+i\theta)n}-e^{(-u+i(\theta-h))n})\langle \Id_{Y_n}[v(e^{-u+i(\theta-h)})-v(1)],1\rangle
\end{align*}

Note that
$\mu(e^{(-u+i\theta)\vf}-e^{(-u+i(\theta-h))\vf})=\int_0^\infty e^{-(u-i\theta)x}(e^{ihx}-1)dG(x)$,
where $G(x)=\mu(\varphi\leq x)$.
Under (H3), the estimate $|\int_0^\infty e^{-(u-i\theta)x}(e^{ihx}-1)dG(x)|\ll h^\beta\ell(1/h)$ follows by the argument
used in the proof of~\cite[Lemma 3.3.2]{GL}.

Next, by the argument used in the proof of  Lemma~\ref{lemma-asymplambda2}, which relies on Lemma~\ref{cor-estKL} and Lemma~\ref{lem:compone}
with  $\ggen\beta>\eps_0$, and
using the fact that $h\leq \min\{|\theta|,u\}$, we obtain that for any $\eps_1>\eps_0$,
\begin{align*}
|\sum_n (e^{(-u+i\theta)n}-1)\langle \Id_{Y_n}[v(e^{-u+i\theta})-v(e^{-u+i(\theta-h)})],1\rangle|&\ll h^{\beta-\eps_1}|u-i\theta|^{\ggen\beta-\eps_1}\ell(1/|u-i\theta|)\\
&\ll  h^{\ggen\beta-\eps_1} |u-i\theta|^{\beta-\epsilon_1} \ell(1/|u-i\theta|).
\end{align*}
Proceeding similarly,
\[
|\sum_n (e^{(-u+i\theta)n}-e^{(-u+i(\theta-h))n})\langle \Id_{Y_n}[v(e^{-u+i(\theta-h)})-v(1)],1\rangle|\ll h^{\ggen\beta-\eps_1}\ell (1/h) |u-i\theta|^{\beta-\epsilon_1}.
\]
Since $\eps_1>\eps_0$ is arbitrary, item i) follows by putting together the above estimates.

[ii)] In the first part, we proceed as in the proof of~\cite[Corollary 6.2]{Terhesiu12} and adapt the last part of the mentioned proof to deal
with the complication that here
we require estimates in $\|.\|_{\cB\to\cB_w}$.

Let $\Delta_T=T(e^{u}e^{i\theta})-T(e^{-u}e^{i(\theta-h)})$. Set
\[
\Delta_{\lambda,P}=(1-\lambda(e^{-u}e^{i\theta}))^{-1}P(e^{-u}e^{i\theta})-(1-\lambda(e^{
-u}e^{i(\theta-h)}))^{-1}P(e^{-u}e^{i(\theta-h)})
\] and
 \[
\Delta_{Q}=(I-R(e^{-u}e^{i\theta})Q(e^{-u}e^{i\theta}))^{-1}Q(e^{-u}e^{i\theta})-(I-R(e^{-u}e^{
i(\theta-h)})Q(e^{-u}e^{i(\theta-h)})^{-1}Q(e^{-u}e^{i(\theta-h)}).
\]
By~\eqref{eq-T(z)} and the fact that $(I-R(z))^{-1}Q(z)=(I-R(z)Q(z))^{-1}Q(z)$,  $\Delta_T= \Delta_{\lambda,P}+ \Delta_{Q}$.
Next,
\begin{align*}
\|\Delta_{\lambda,P}\|_{\cB\to\cB_w} &\ll
\|(1-\lambda(e^{-u}e^{i\theta}))^{-1}(P(e^{-u}e^{i\theta})-P(e^{-u}e^{i(\theta-h)}))\|_{\cB\to\cB_w}\\
&+\|P(e^{-u}e^{i(\theta-h)})\Big((1-\lambda(e^{-u}e^{i\theta}))^{-1}-(1-\lambda(e^{-u}e^{
i(\theta-h)}))^{-1}\Big)\|_{\cB\to\cB_w}.
\end{align*}
By  Lemma~\ref{lemma-asymplambda2}, for $u,|\theta|<\delta_0$ we have  $|(1-\lambda(e^{-u+i\theta}))^{-1}|\ll
\ell(1/|u-i\theta|)^{-1}|u-i\theta|^{-\beta}$. This together with
Lemma~\ref{cor-estKL} yields
$\|(1-\lambda(e^{-u}e^{i\theta}))^{-1}(P(e^{-u}e^{i\theta})-P(e^{-u}e^{i(\theta-h)}))\|_{\cB\to\cB_w}\ll
\ell(1/|u-i\theta|)^{-1}h^{\beta-\epsilon_1}|u-i\theta|^{-\beta}$. By (i) of the present Proposition and  Lemma~\ref{lemma-asymplambda2},
\begin{align*}
\bigg\|\frac{P(e^{-u}e^{i(\theta-h)})(\lambda(e^{-u}e^{
i(\theta-h)})-\lambda(e^{-u}e^{i\theta}))}{((1-\lambda(e^{-u}e^{i\theta}))(1-\lambda(e^{-u}e^{
i(\theta-h)})} &\bigg\|_{\cB\to\cB_w}\ll
\ell(1/|u-i\theta|)^{-2}\ell(1/h)h^{\beta}|u-i\theta|^{-2\beta}\\
&+\ell(1/|u-i\theta|)^{-1}h^{\ggen\beta-\epsilon_1}|u-i\theta|^{-\beta}.
\end{align*}
 Thus,
\begin{align*}
\|\Delta_{\lambda,P}\|_{\cB\to\cB_w}&\ll
\ell(1/|u-i\theta|)^{-1}h^{\ggen\beta-\epsilon_1}|u-i\theta|^{-\beta}
+\ell(1/|u-i\theta|)^{-2}\ell(1/h)h^{\beta}|u-i\theta|^{-2\beta}.
\end{align*}
To estimate $\|\Delta_{Q}\|_{\cB\to\cB_w}$, we compute that
\begin{align*}
\|\Delta_{Q}\|_{\cB\to\cB_w}\ll \| D(u,\theta,h)\|_{\cB\to\cB_w}+\| E(u,\theta,h)\|_{\cB\to\cB_w}
\end{align*}
where
 $D(u,\theta,h)=(I-R(e^{-u}e^{i\theta})Q(e^{-u}e^{i\theta}))^{-1}F(u,\theta,h)$
and
$E(u,\theta,h)=(I-R(e^{-u}e^{i\theta})Q(e^{-u}e^{i\theta}))^{-1}G(u,\theta,h)$
with
\begin{align*}
F(u,\theta,h)=(Q(e^{-u}e^{i\theta})-Q(e^{-u}e^{i(\theta-h)}))
=(Q(e^{-u}e^{i\theta})-Q(e^{-u}e^{i(\theta-h)}))(Q(e^{-u}e^{i\theta})+Q(e^{-u}e^{i(\theta-h)})).
\end{align*}
and
\begin{align*}
G(u,\theta,h) =[(RQ)(e^{-u}e^{i\theta})-(RQ)(e^{-u}e^{i\theta-h})]
 (I-R(e^{-u}e^{i(\theta-h)})Q(e^{-u}e^{i(\theta-h)}))^{-1}Q(e^{-u}e^{i(\theta-h)})
\end{align*}

By (H1)(iv) and (H5)(i), there exists  $C>0$ such that for $z\in\bar\D\setminus B_{\delta_0}(1)$,
$\|(I-R(z)Q(z))^{-1}\|_{\cB}\leq C$.
So, $\max\{\|F(u,\theta,h)\|_\cB,\, \|G(u,\theta,h)\|_\cB\}\leq C$ for some constant $C>0$. This together with Lemma~\ref{cor-estKL}
implies that for any $\eps_1>\eps_0$, there exists $ C_{\eps_1}>0$ such that
\begin{equation}\label{eq-FG}
\max\{\| F(u,\theta,h)\|_{\cB_w},\, \|G(u,\theta,h)\|_{\cB_w}\}\leq C_{\eps_1} h^{\beta-\eps_1}
\end{equation}

Recall that $\|(R(z)Q(z))^{n}\|_{\cB}\leq C\sigma^n$ and $\|(R(z)Q(z))^{n}\|_{\cB_w}\leq C$.
Let $A(z)$, $z\in\bar\D\cap B_{\delta_0}(1)$ be a family of operators
well defined in both spaces $\cB$ and $\cB_w$.
Using arguments in smiler to the ones in~\cite{KellerLiverani99}, these inequalities imply that for any $v\in\cB$, 
\begin{equation}\label{eq-KLcalc}
\begin{split}
\|(I -R(z)Q(z))^{-1}A(z)v\|_{\cB_w}&\leq\sum_{k=0}^{n-1}\|(R(z)Q(z))^k A(z)v\|_{\cB_w}\\
&\quad+\|(R(z)Q(z))^n(I-R(z)Q(z))^{-1} A(z)v\|_{\cB}\\
 &\leq n C\|A(z)v\|_{\cB_w}+C\sigma^n\|A(z)v\|_{\cB},
\end{split}
\end{equation}

Taking $n=[(\beta-\eps_0)\log(h)\log(1/\sigma)$] (so $\sigma^{n}\ll h^{\beta-\eps_0}$ and $n\ll \log (1/h)$), applying the above inequality to $A=F,\, G$ and using ~\eqref{eq-FG}, we obtain
that for any $\eps_1>\eps_0$
\[
\|\Delta_{Q}\|_{\cB\to\cB_w} \ll h^{\beta-\epsilon_0}\log(1/h)\ll
 h^{\beta-\epsilon_1}.
\] 
Item ii) follows by putting together the estimates for $\|\Delta_{\lambda,P}\|_{\cB\to\cB_w}$
and $\|\Delta_{Q}\|_{\cB\to\cB_w}$.~\end{proof}

\begin{rmk}In the case $\beta=1$,
the scheme above can also be combined with the arguments in~\cite{MT}, providing the desired asymptotics of $T(z)$,
$z\in\mathcal{S}^1$ and as such, first and higher order theory for the coefficients $T_n$ of $T(z)$,  $z\in\bar\D$.
To simplify the exposition in what follows we omit the case $\beta=1$.
\end{rmk}

The bounds provided by Proposition~\ref{prop-estTdifKL} do not allow one to apply
directly the argument of~\cite{MT} to estimate of the coefficients of $T(z)$, $z\in\bar\D$:
the arguments in~\cite{MT} require that
$\|T(e^{-u+i\theta})-T(e^{-u+i(\theta-h)})\|_{\cB\to\cB_w}\ll\ell(1/|u-i\theta|)^{-2}
\ell(1/h)h^{\beta}|u-i\theta|^{-2\beta}$. However, as explained in Section~\ref{sec-FO-Tn}, the modified version of
these arguments in~\cite{Terhesiu12} applies. In this sense, we note that

\begin{prop}\label{prop-derivR}Assume (H4)(ii). Write $z=e^{-(u+i\theta)}$. Then for all $u>0$,
\[
\left\|\frac{d}{d\theta}(R(z))\right\|_{\cB\to\cB_w}\ll u^{\beta-\epsilon_0-1}.
\]
\end{prop}
\begin{proof}The result follows by the argument used in the proof
of~\cite[Proposition 4.6]{Terhesiu12}.~\end{proof}

As a consequence we have
\begin{cor}\label{cor-deriv} Assume (H1), (H3),  (H4)(ii) and (H5). Write $z=e^{-(u+i\theta)}$.
Choose $\delta_0$ such that Lemma~\ref{cor-estKL} holds.
Let $\theta$ such that $|\theta|>\delta_0$.
Then for all $u>0$,
\[
\left\|\frac{d}{d\theta}(I-R(e^{-u}e^{i\theta}))^{-1}\right\|_{\cB\to\cB_w}\ll u^{\beta-\epsilon_0-1}\log(1/u).
\]
\end{cor}
\begin{proof}By (H5) there exists $\delta>0$ and some constant $C>0$ such that
$\|(I-R(e^{-u}e^{i\theta}))^{-1}\|_{\cB}\leq C$, for all $u>0$ and $|\theta|>\delta$. This together with
Remark~\ref{rmk-normw} ensures that $\|(I-R(e^{-u}e^{i\theta}))^{-1}\|_{\cB}\leq C$, for all $u>0$ and $|\theta|>\delta_0$. 

Note that $\frac{d}{d\theta}(I-R(z))^{-1}=(I-R(z))^{-1}\frac{d}{d\theta}R(z)(I-R(z))^{-1}$.
Let $h\in\cB$. Proceeding as in~\eqref{eq-KLcalc} and using (H5) (i), we obtain that there exists
some constant $C>0$ such that
\[
\|\frac{d}{d\theta}(I-R(z))^{-1}h\|_{\cB_w}\leq n C\|\frac{d}{d\theta}R(z)(I-R(z))^{-1}h\|_{\cB_w}+C\lambda^{-n}\| \frac{d}{d\theta}R(z)(I-R(z))^{-1}h\|_{\cB}.
\]
Using the fact that $\|R_n\|_\cB\leq C$, for all $n\geq 1$ and some constant $C>0$, it is easy to check that $\left\|\frac{d}{d\theta}(R(z))\right\|_{\cB\to\cB_w}\ll u^{-2}$.
Take $n=[3\log(u)\log(1/\lambda)^{-1}$]; with this choice so $\lambda^{-n}\ll u^3$ and $n\ll \log(1/u)$. Hence, $\lambda^{-n}\| \frac{d}{d\theta}R(z)(I-R(z))^{-1}h\|_{\cB}\ll u$
and by Proposition~\ref{prop-derivR}, $\|\frac{d}{d\theta}R(z)(I-R(z))^{-1}h\|_{\cB_w}\ll u^{\beta-\epsilon_0-1}\log(1/u)$.
The conclusion follows from these estimates and the last displayed inequality.~\end{proof}

\section{First order asymptotic of $T_n$: mixing.}
\label{sec-FO-Tn}

Given the asymptotic behaviors of  $T(z):\cB\to\cB_w$ for $z\in \bar\D\cap B_\delta(1)$ and of
$\frac{d}{d\theta}(T(z))$ for $z\in \bar\D\setminus B_\delta(1)$ described  in Section~\ref{sec-FO_Tz},
the arguments used in~\cite{Terhesiu12}(a modified version of~\cite{MT}) for estimating  the coefficients $T_n$
of $T(z)$,  $z\in\bar\D$ apply. We briefly recall the main steps.

By, for instance, the argument of~\cite[Corollary 4.2]{MT},

\begin{lemma}\label{lemma-id} Let $A(z)$ be a function from $\bar\D$ to some Banach space $\widetilde \cB$, continuous on $\bar\D\setminus \{1\}$
and analytic  on $\D$.
  For $u\geq 0$, $\theta\in [-\pi,\pi)$, write $z=e^{-u+i\theta}$. Assume that
as $z\to 1$,
\[
|A(e^{-u+i\theta})|\ll|A(e^{i\theta})|\ll|\theta|^{-\gamma},
\]
for some $\gamma\in (0,1)$. Then the Fourier coefficients $A_n$ coincide with the Taylor coefficients $\hat{A}_n $, that is
\[
A_n=\hat{A}_n=\frac{1}{2\pi}\int_{-\pi}^{\pi}A(e^{i\theta}) e^{-in\theta}\,d\theta
\]
\end{lemma}

\begin{cor}\label{rem-TaylorForurier-scalar}Assume (H1), (H2),  (H3), (H4)(ii) and (H5). Then,
the Taylor coefficients of $T(z):\cB\to\cB_w$, $z\in\D$ coincide with the Fourier coefficients
of  $T(z):\cB\to\cB_w$, $z\in\mathcal{S}^1$.
\end{cor}
\begin{proof} This is an immediate consequence of Lemma~\ref{lemma-id} and Corollary~\ref{cor-estTdifKL}.~\end{proof}

 By Corollary~\ref{rem-TaylorForurier-scalar}, first and higher order of $T_n$  can be obtained by estimating
either the Fourier or Taylor coefficients of $T(z)$, $z\in\bar\D$.

\subsection{Mixing under (H4)(ii)}

\begin{thm} \label{lemma-FO-Tn-MT}
Assume (H1), (H3), (H4)(ii) and (H5). Assume (H2) with $\ggen\beta>\eps_0$ and
$\beta\in (1/2, 1)$. Let $d_0=[\Gamma(1-\beta)\Gamma(\beta)]^{-1}$.  Then, as $n\to\infty$,
\[
\sup_{v\in\cB, \|v\|_{\cB}=1}\|\ell(n)n^{1-\beta}T_n v-
d_0 Pv\|_{\cB_w}\to 0.
\]
\end{thm}

\begin{rmk}The above result generalizes~\cite[Theorem 2.1]{MT} to the abstract class of transformations described
in Section~\ref{sec-oprenseq}.
\end{rmk}

\begin{proof}We argue as in~\cite[Proof of Theorem 3.3]{Terhesiu12}.

Let $\Gamma_{1/u}=\{e^{-u}e^{i\theta}:-\pi\leq \theta <\pi\}$
with $e^{-u}=e^{-1/n}$, $n\geq 1$.
Let $b\in (0,\delta_0 n)$, with $\delta_0$ as in Lemma~\ref{cor-estKL}.
Let $A=[-\pi, -\delta_0]\cup [\delta_0, \pi]$.

With the above specified, we proceed to estimate $T_n$.

\begin{align}\label{eq-Tn-int}
\nonumber T_n&=\frac{1}{2\pi}\int_{\Gamma_n} \frac{T(z)}{z^{n+1}} dz=\frac{e}{2\pi}\int_{-\pi}^{\pi}
T(e^{-1/n}e^{i\theta})e^{-in\theta}d\theta=
\frac{e}{2\pi}\Big(\int_{-b/n}^{b/n} T(e^{-1/n}e^{i\theta})e^{-in\theta}d\theta\\
\nonumber &+\int_{-\delta_0}^{-b/n}
T(e^{-1/n}e^{i\theta})e^{-in\theta}d\theta+\int_{b/n}^{\delta_0} T(e^{-1/n}e^{i\theta})e^{-in\theta}d\theta+
\int_A T(e^{-1/n}e^{i\theta})e^{-in\theta}d\theta\Big)\\
&=\frac{e}{2\pi}\int_{-b/n}^{b/n} 
T(e^{-1/n}e^{i\theta})e^{-in\theta}d\theta+\frac{e}{2\pi}(I_{\delta_0}+I_{-\delta_0}+I_A).
\end{align}

By~\cite[Proposition 4.6]{Terhesiu12}, 
\[
\lim_{b\to\infty}\lim_{n\to\infty}n^{1-\beta}
\ell(n)\Gamma(1-\beta)\int_{-b/n}^{b/n}
T(e^{-1/n}e^{i\theta})e^{-in\theta}d\theta=\frac{2\pi}{e}\frac{1}{\Gamma(\beta)}P.
\]
Hence, the conclusion will follow once we show that $n^{1-\beta}\ell(n)I_A=o(1)$ and 
 $\lim_{b\to\infty}\lim_{n\to\infty}
 n^{1-\beta}\ell(n)(I_{\delta_0}+I_{-\delta_0})=0$.
 
We first estimate $I_A$. Compute that 
\[
I_A=\frac{i}{n}\int_A T(e^{-1/n}e^{i\theta})\frac{d}{d\theta}(e^{-in\theta})d\theta=
\frac{1}{in}\int_A \frac{d}{d\theta}(T(e^{-1/n}e^{i\theta}))e^{-in\theta}d\theta+E(n),
\] 
where $E(n)\ll n^{-1}(\|T(e^{-1/n}e^{i\epsilon b/n})\|_{\cB\to\cB_w}+
\|T(e^{-1/n}e^{i\pi})\|_{\cB\to\cB_w})$. By Corollary~\ref{cor-estTdifKL} and (H5),
$\|T(e^{-1/n}e^{i\theta})\|_{\cB\to\cB_w}=O(1)$ for all $\theta\in A$.
Hence $E(n)=O(n^{-1})$. Note that for $\theta\in A$, $|\theta|>\delta_0$ and thus
 Corollary~\ref{cor-deriv} applies.
It follows that
$\|\frac{d}{d\theta}(T(z))\|_{\cB\to\cB_w}\ll n^{1-\beta+\epsilon_0}\log n$.
Putting these together, 
\begin{equation}\label{eq-est-IA}
|I_A|\ll n^{-(\beta-\epsilon_0)}\log n+n^{-1}\ll
n^{-(\beta-\epsilon_0)}\log n.
\end{equation}
 Since  $\epsilon_0<\beta^*$ , where $\beta^*<\max\{2\beta-1,1-\beta\}$,
we have  $n^{1-\beta}\ell(n)|I_A|\ll n^{-(2\beta-1-\epsilon_0)}\ell(n)\log n=o(1)$.

Next, we estimate $I_{\delta_0}$. The estimate for $I_{-\delta_0}$ follows by a similar argument.
Recall $b\in (0,n\delta_0)$. Proceeding as in the proof of ~\cite[Lemma 5.1]{MT}(see also~\cite{GL}),  we write
\[
I_{\delta_0}=\int_{b/n}^{\delta_0} T(e^{-1/n}e^{i\theta}) e^{-in\theta}\,d\theta
=-\int_{(b+\pi)/n}^{\delta_0+\pi/n}T(e^{-1/n}e^{i(\theta-\pi/n)}) e^{-in\theta}\,d\theta.
\]
Hence 
\[
2I_{\delta_0}=\int_{b/n}^{\delta_0} T(e^{-1/n}e^{i\theta}) e^{-in\theta}\,d\theta
-\int_{(b+\pi)/n}^{\delta_0+\pi/n}T(e^{-1/n}e^{i(\theta-\pi/n)}) e^{-in\theta}\,d\theta=I_1+I_2+I_3,
\]
where
\begin{align*}
I_1  & = \int_{\delta_0}^{\delta_0+\pi/n}T(e^{-1/n}e^{i(\theta-\pi/n)}) e^{-in\theta}\,d\theta, \qquad
I_2  = \int_{b/n}^{(b+\pi)/n}T(e^{-1/n}e^{i(\theta-\pi/n)}) e^{-in\theta}\,d\theta, \\
I_3 & =\int_{(b+\pi)/n}^{\delta_0} \{  T(e^{-1/n}e^{i\theta})-T(e^{-1/n}e^{i(\theta-\pi/n)})\}e^{-in\theta}\,d\theta.
\end{align*}

We already know $\|T(e^{-1/n}e^{i\theta})\|_{\cB}=O(1)$ for all $|\theta|>\delta_0$. Thus, $|I_1|\ll 1/n$. By Corollary~\ref{cor-estTdifKL},
 $\|T(e^{-u}e^{i\theta})\|_{\cB\to\cB_w}\ll
\ell(1/|u-i\theta|)^{-1}|u-i\theta|^{-\beta}$. This together with standard calculations implies that
$|I_2|\ll \ell(n)^{-1}n^{-(1-\beta)} b^{-(\beta-\gamma^*)}$,
  for any $0<\gamma^*<\beta$.  Putting the above together,
 $n^{1-\beta}\ell(n)I_{\delta}=n^{1-\beta}\ell(n)I_3+O( b^{-(\beta-\gamma^*)})$. 
 
 Next, we estimate $I_3$. By Proposition~\ref{prop-estTdifKL}, for all $\theta\in ((b+\pi)/n, \delta_0)$ and
 for any $\epsilon_1\in (\epsilon_0,2\beta-1)$, we have 
\begin{align*}
 \|T(e^{-1/n}e^{i\theta})-T(e^{-1/n}e^{i(\theta-\pi/n)})\|_{\cB\to\cB_w}&\ll 
\ell(n)\ell(n/|1-in\theta|)^{-2}n^{-\beta}|\frac 1n-i\theta|^{-2\beta}\\
&+\ell(n/|1-in\theta|)^{-1}
n^{-(\ggen\beta-\epsilon_1)} |\frac 1n-i\theta|^{-\beta}.
\end{align*}
 Hence, 
\begin{align}\label{eq_I3}
\nonumber|I_3| &\ll n^{-\beta}\ell(n)\int_{(b+\pi)/n}^{\delta_0} \ell(n/|1-in\theta|)^{-2}\theta^{-2\beta}d\theta
+n^{-(\ggen\beta-\epsilon_1)}\int_{(b+\pi)/n}^{\delta_0}
\ell(n/|1-in\theta|)^{-1}\theta^{-\beta}d\theta
\\\nonumber& 
\ll \ell(n)^{-1}n^{-\beta}\int_{(b+\pi)/n}^{\delta_0} \theta^{-2\beta}
\frac{\ell(n)^{2}}{\ell(n/|1-in\theta|)^{2}}d\theta
+\ell(n)^{-1}n^{-(\ggen\beta-\epsilon_1)}\int_{(b+\pi)/n}^{\delta_0}
\theta^{-\beta}\frac{\ell(n)}
{\ell(n/|1-in\theta|)}d\theta\\
&= \ell(n)^{-1}n^{-\beta}I_{3,1}+
\ell(n)^{-1}n^{-(\ggen\beta-\epsilon_1)}I_{3, 2}.
\end{align}

Using Potter's bounds (see, for instance,~\cite{BGT}), for any $\delta_1>0$,
\[
I_{3,1}=n^{2\beta-1} \int_{b+\pi}^{n\delta_0} \sigma^{-2\beta}\frac{\ell(n)^{2}}{\ell(n/|1-i\sigma|)^{2}}\,d\sigma
\ll n^{(2\beta-1)}\int_{b+\pi}^{n\delta_0} \sigma^{-(2\beta-\delta_1)}d\sigma.
\]
Taking $0<\delta_1<2\beta -1$, 
\begin{equation*}\label{eq_I31}
\ell(n)^{-1}n^{-\beta}|I_{3,1}|\ll \ell(n)^{-1}n^{-\beta}n^{2\beta-1} b^{2\beta-\delta_1-1}=
\ell(n)^{-1} n^{\beta-1}b^{-(2\beta-1-\delta_1)}. 
\end{equation*}

Finally, we estimate $I_{3,2}$. Using Potter's bounds,
we obtain that for any $\delta_2>0$ and $\delta_3>0$,  
\[
|I_{3,2}|\ll \int_{(b+\pi)/n}^{\delta_0} \theta^{-\beta}
\Big((\theta n)^{\delta_2}+(\theta n)^{-\delta_3}\Big)\, d\theta\ll
n^{\delta_2}\int_{(b+\pi)/n}^{\delta_0} \theta^{-(\beta+\delta_3)}\, d\theta.
\]
Hence, $\ell(n)^{-1}n^{-(\ggen\beta-\epsilon_1)}|I_{3,2}|\ll \ell(n)^{-1}n^{-(\ggen\beta-\epsilon_1-\delta_3)}$
for arbitrary small $\delta_2$.

 Putting together the estimates for $I_1, I_2$ and $I_3$ (using~\eqref{eq_I3} and the estimates for
$\ell(n)^{-1}n^{-\beta}|I_{3,1}|$ and  $\ell(n)^{-1}n^{-(\ggen\beta-\epsilon_1)}|I_{3,2}|$), we obtain that
for arbitrary small $\delta_1,\delta_3$,
\begin{equation}\label{eq-est-Ieps}
|I_{\delta_0}|\ll n^{\beta-1}\ell(n)^{-1} b^{-(2\beta-1-\delta_1)}+n^{-(\ggen\beta-\epsilon_1-\delta_3)}\ell(n)^{-1}.
\end{equation} 
Since $\ggen\beta>\eps_0$ and  $\eps_0<\epsilon_1<\beta^*$ with $\beta^*<\max\{2\beta-1,1-\beta\}$,
\begin{equation*}
|I_{\delta_0}|\ll n^{\beta-1}\ell(n)^{-1} b^{-(2\beta-1-\delta_1)}.
\end{equation*} 
Hence, $n^{1-\beta}\ell(n)|I_{\delta}|\ll b^{-(2\beta-1-\delta_1)}$ and thus,
$\lim_{b\to\infty}\lim_{n\to\infty} n^{1-\beta}\ell(n)I_{\delta_0}=0$. By a similar argument,
 $\lim_{b\to\infty}\lim_{n\to\infty}n^{1-\beta}\ell(n)|I_{-{\delta_0}}|=0$, ending the proof.~\end{proof}

We can now complete

\begin{pfof}{Theorem~\ref{cor-equiMTl}} The conclusion follows from  Theorem~\ref{lemma-FO-Tn-MT} and  Remark~\ref{rmk-Ident}.
For completeness, we recall the standard argument.

Recall $P1=h$,  $P1=\mu$ and $Pv=h\langle v,1\rangle$.
Assumption (H1)(ii) ensures that for any $C^\alpha$ observable $v:M\to\R$, $v$
supported on $Y$, we have $P(v h)=(\int_M v\, d\mu)h$.
Also, by Theorem~\ref{lemma-FO-Tn-MT} and Lemma~\ref{lemma-FO-Tn-G},
$\|\ell(n)n^{1-\beta}\Id_Y L^n (vh)- d_0 Pvh\|_{\cB\to\cB_w}=o(1)$.
Putting these together,
\begin{align*}
\ell(n)n^{1-\beta}\int_M  v\, w\circ f^n \, d\mu &=
\ell(n)n^{1-\beta}\langle L^n (vh), w\rangle=d_0\langle  P(vh),w\rangle +o(1)\\
& =d_0 \Big(\int_M v\, d\mu\Big) \langle h,w\rangle  dm+o(1)=d_0 \int_M v\, d\mu\int_M w\, d\mu+o(1).
\end{align*}~\end{pfof}

\subsection{Mixing under (H4)(iii) in the infinite case with $\beta\in (0,1)$}

We recall that under hypotheses (H1), (H2) with $\ggen=1$, (H3), (H4)(i) and (H5)(i),  Lemma~\ref{lemma-asymplambda0} gives
precise information about the asymptotic behavior of $1-\lambda(z)$. Repeating the argument used in the proof
of  Lemma~\ref{lemma-asymplambda0}
with (H4)(iii) instead of (H4)(i) and (H2) with $\ggen\in (0,1]$, we obtain that the conclusion
of  Lemma~\ref{lemma-asymplambda0} under (H4)(iii) and (H2) with $\ggen\in (0,1]$.
Given this,  the arguments in~\cite{Gouezel11} carry over
with no modification, yielding 

\begin{lemma}{A consequence of~\cite[Theorem 1.4]{Gouezel11}} \label{lemma-FO-Tn-G}
Assume that (H1), (H2) with $\ggen\in (0,1]$ and (H5) hold. Let $\beta\in (0,1)$ and suppose that (H3) and (H4)(iii) hold. Then,
\[
\sup_{v\in\cB, \|v\|_{\cB}=1}\|\ell(n)n^{1-\beta}T_n v-d_0 Pv\|_{\cB}\to 0.
\] 
\end{lemma}

By the argument used in the proof of Theorem~\ref{cor-equiMTl} with Theorem~\ref{lemma-FO-Tn-MT} replaced by Lemma~\ref{lemma-FO-Tn-G}, we obtain 

\begin{cor} \label{cor-equilG}
Assume the setting of Lemma~\ref{lemma-FO-Tn-G}. Then Theorem~\ref{cor-equiMTl} holds
 for all $\beta\in (0,1)$.
\end{cor}

\subsection{Polynomial decay of correlation under (H4)(iii) in the finite case with $\beta>1$}
\label{sec-decay}
We note that in the present framework under (H1), (H4)(iii) with $\beta>1$ and (H5), the abstract ~\cite[Theorem 1]{Gouezel04}
holds. As a consequence of this result and using the argument used in the proof of  Theorem~\ref{cor-equiMTl} with Theorem~\ref{lemma-FO-Tn-MT} replaced by~\cite[Theorem 1]{Gouezel04}, one has

\begin{cor} \label{cor-dec}{A consequence of~\cite[Theorem 1]{Gouezel04}.}
Assume (H1) and (H5).  Suppose that $\mu(\varphi>n)=O(n^{-\beta})$  with $\beta>1$
and assume that  (H4)(iii) holds. Let $c_0n^{-(\beta+1)}=\sum_{k>n}\mu(\varphi=k)$. 
 If $v,w:M\to\R$ are $C^\alpha$ (with $\alpha$ as in (H1)(i)) observables supported on $Y$, then
\[
\lim_{n\to\infty}n^{1-\beta}\Big|\int_M  v\, w\circ f^n \, d\mu - \int_M v\, d\mu\int_M w\, d\mu\Big|=c_0  \int_M v\, d\mu\int_M w\, d\mu.
\]
\end{cor}

\subsection{Weak pointwise dual ergodicity under weak assumptions}
\label{rmk-wpde}
As mentioned in the introduction, 
the present framework allows us to deal with  the property of weak
pointwise dual ergodicity (weak p.d.e.) under some weak conditions (under which mixing cannot be proved). Below,
we provide a result that allows one to check  weak p.d.e. in the framework of  Section~\ref{sec-oprenseq} without assuming
(H5) and only requiring (H4)(i) and (H5)(i).

 We recall that a conservative ergodic measure preserving transformation $(X,\mathcal{A}, f,\mu)$
is pointwise dual ergodic (p.d.e.)  if there exists some positive sequence $a_n$
such that $\lim_{n\to\infty}a_n^{-1}\sum_{j=0}^nL^jv= \int_X v\, d\mu$,  a.e. on $X$  for all $v\in L^1(\mu)$.
The property of weak p.d.e. has been recently exploited and defined in~\cite{AZ}.
As noted in~\cite{Aaronson},  if $f$is invertible and $\mu(X)=\infty$ then $f$ cannot be p.d.e.,
but it can be weak p.d.e.; that is, there exists some positive sequence $a_n$ such that
\begin{itemize}
\item[(i)] $a_n^{-1}\sum_{j=0}^nL^jv\to^{\nu} \int_X v\, d\mu$ as $n\to\infty$, for all $v\in L^1(\mu)$. Here,
 $\to^{\nu}$ stands for convergence
in measure for any finite measure $\nu\ll\mu$.
\item[(ii)]$\limsup_{n\to\infty}a_n^{-1}\sum_{j=0}^nL^jv= \int_X v\, d\mu$, a.e. on $X$  for all $v\in L^1(\mu)$.
\end{itemize}

As shown in~\cite[Proposition 3.1]{AZ}, weak p.d.e. for infinite c.e.m. p.t. can be established as soon as items
(i) and (ii) above are shown to hold for $v=1_Y$ for some $Y\in\mathcal{A}$ with $0<\mu(Y)<\infty$. Moreover,
as noted elsewhere (see~\cite{AZ} and reference therein), item (i) follows as soon as the above mentioned 
convergence  in measure is established for $\mu|_Y$ for some $Y\in\mathcal{A}$ with $0<\mu(Y)<\infty$.

In the  framework of  Section~\ref{sec-oprenseq},  the following holds result for original transformations $f$
with first return map $F:Y\to Y$:

\begin{prop}\label{prop-wpde}Assume (H1), (H2), (H3), (H4)(i) and (H5)(i). Furthermore, set $a_n=\ell(n)n^{1-\beta}d_0$ and
 suppose that $\sup_{Y}a_n^{-1}\sum_{j=0}^nL^j 1_Y\to 1$, $\mod \mu$,  as $n\to\infty$. Then $f$ is weak p.d.e.
\end{prop}

\begin{proof}
Recall that under  (H1), (H2), (H3) , (H4)(i) and (H5)(i), equation~\eqref{eq-asymp0}
holds. By argument used in  Corollary~\ref{eq-asymp0}, 
\[
\|T(e^{-u})-d_0\ell(1/|u)^{-1}u^{-\beta} P\|_{\cB}\to 0,
\mbox{ as } u\to 0.
\]
Let $\nu$ be a $\sigma$-finite measure on $M$ such that $\nu$ is supported on $Y$
and $\nu|_Y$ is a probability measure in $\cB$. Recall $P1=\mu$ and $\mu(Y)=1$.

By assumption, $\sup_{Y}\ell(n)n^{1-\beta}d_0\sum_{j=0}^nL^j 1_Y=O(1)$,  as $n\to\infty$. Thus,
 the argument used in the proof of~\cite[Lemma 3.5]{MT11}( a  Karamata Tauberian theorem for positive operators
that generalizes Karamata tauberian theorem for scalar sequences~\cite[Proposition 4.2]{BGT}) applies. It follows that
\[
\lim_{n\to\infty}\ell(n)n^{1-\beta}\sum_{j=0}^{n-1} L^{j^*}\nu(1_Y) =d_0.
\]
In particular, the above equation holds for $\nu=\mu$. So,
$\ell(n)n^{1-\beta}d_0\sum_{j=0}^{n-1}L^{j^*} \mu(1_Y)\to 1$, as $n\to\infty$ and item (i) ( in the definition of weak p.d.e.) for the function
$1_Y$ follows.~\end{proof}

\section{Higher order asymptotic of $T_n$: mixing rates}
\label{sec-HO-Tn}

As already mentioned in the introduction, mixing rates for non-invertible infinite measure preserving
systems have been obtained in~\cite{MT, Terhesiu12}.
The results in these works depend heavily on a higher order expansion of  the
tail probability $\mu(\varphi>n)$. The arguments in~\cite{MT, Terhesiu12} generalize to set up of
Section~\ref{sec-oprenseq} and (in an obvious notation), we state

\begin{thm} \label{lemma-HOTn}
Assume (H1),  (H3), (H4)(ii) and (H5). Let $q=\max\{ j\geq 0: (j+1)\beta-j>0\}$. Then there exist real
constants
$d_0,\ldots, d_q$ (depending only on the map $f$),\footnote{ For the precise form of these constants we refer to  in~\cite[Theoreme 9.1]{MT}
and~\cite[Theoreme 1.1]{Terhesiu12}.} such that the following hold for any $\eps_2>\eps_1>\eps_0$:
\begin{itemize}
\item[(i)] Assume (H2) with $\ggen=1$. Suppose that $\mu(\varphi>n)=c n^{-\beta}+H(n)$
for some $c>0$ and $H(n)=O(n^{-2\beta})$.

Then,
\[
T_n =(d_0 n^{\beta-1}+d_1 n^{2\beta-2}+\ldots+d_q n^{(q+1)(\beta-1)})P+D,
\] where $\|D\|_{\cB\to\cB_w}=O(\max\{n^{-(\beta-\eps_2)}, n^{-(1-\beta+\eps_1)/2}\})$.

\item[(ii)] Assume (H2) with  $\ggen=1$.   Let $\beta>1/2$ and suppose that
  $\mu(\varphi>n)=cn^{-\beta}+b(n)+H(n)$, for some  $c>0$, some function $b$ such that 
$nb(n)$ has bounded variation and $b(n)=O(n^{-2\beta})$,  and some function $H$ such that 
$H(n)=O(n^{-\gamma})$ with $\gamma>2$. 
Then (i) holds with the improved rate $\|D\|_{\cB\to\cB_w}=O( n^{-\eps_1})$.

\item[(iii)] Suppose that $\mu(\varphi>n)=c n^{-\beta}+H(n)$
for some $c>0$ and $H(n)=O(n^{-2\beta})$. Assume (H2) with $\ggen\beta>\eps_0$.
Then (i) holds with the  rate $\|D\|_{\cB\to\cB_w}=O(\max\{n^{-(\ggen\beta-\eps_2)}, n^{-(1-\ggen\beta+\eps_1)/2}\})$.

\end{itemize}
\end{thm}

\begin{rmk} \label{rmk-tailf}
We note that items i), ii) correspond to the results on mixing rates for non-invertible systems
provided by~\cite[Theorem 9.1]{MT} and~\cite[Theorem 3.1]{Terhesiu12}, respectively.
If instead of (H4)(ii) we assume (H4)(iii) in the statement of Theorem~\ref{lemma-HOTn} (with the rest of the assumptions unchanged)
then item i) holds with the improved rate $\|D\|_{\cB\to\cB_w}=O(n^{-(\beta-1/2)})$
and item ii) holds with the improved rate $\|D\|_{\cB\to\cB_w}=O(n^{-\beta})$. Under (H4)(iii), no proof is required:
the arguments used in the proofs of~\cite[Theorem 9.1]{MT} and~\cite[Theorem 3.1]{Terhesiu12} simply go through.
\end{rmk}

\begin{proof}
Below we provide the argument for item iii). Item i) is proved during the proof for item iii). Item ii) follows by the argument used in the proof of ~\cite[Theoreme 1.1]{Terhesiu12}(essentially, by~\cite[Remark 2.7]{Terhesiu12}
and the argument used there);
while the estimates obtained in the proof of ~\cite[Theoreme 1.1]{Terhesiu12} change due to the present assumption (H4)(ii), the argument goes almost word by word.

 Choose  $\delta_0$ such that Lemma~\ref{cor-estKL} holds. Let $b\in (0,\delta n)$, $n\geq 1$.
Also, let $I_A, I_{\delta_0}$ and $I_{-\delta_0}$ be
defined as in the proof of Theorem~\ref{lemma-FO-Tn-MT}. Hence, equation~\eqref{eq-Tn-int} holds and we can write
\[
T_n=\frac{e}{2\pi}\int_{-b/n}^{b/n} 
T(e^{-1/n}e^{i\theta})e^{-in\theta}d\theta+\frac{e}{2\pi}(I_{\delta_0}+I_{-\delta_0}+I_A).
\]

 Let $\epsilon_0$ be as defined in assumption (H4)(ii).
By equation~\eqref{eq-est-IA}, $|I_A|\ll n^{-(\beta-\epsilon_0)}\log n$.

By equation~\eqref{eq-est-Ieps},
for any $\delta_1,\delta_3>0$ and $\eps_1>\eps_0$
\[
|I_{\delta_0} +I_{-\delta_0}|\ll n^{\beta-1}\ell(n)^{-1} b^{-(2\beta-1-\delta_1)}+n^{-(\ggen\beta-\epsilon_1-\delta_3)}\ell(n)^{-1}.
\]

Next, we estimate $\frac{e}{2\pi}\int_{-b/n}^{b/n} 
T(e^{-1/n}e^{i\theta})e^{-in\theta}d\theta$.
Under (H2) with $\ggen\leq 1$ and (H4) (ii), Lemma~\ref{lemma-asymplambda2} gives that
$|(1-\lambda((e^{-1/n}e^{i\theta}))^{-1}|\ll |\frac 1n-i\theta|^{-\beta}$.
This  together with Corollary~\ref{cor-estKL} implies that for any $\eps_1>\eps_0$,
\[
\|(1-\lambda(e^{-1/n}e^{i\theta}))^{-1}(P(e^{-1/n}e^{i\theta})-P)\|_{\cB\to\cB_w}\ll
|\frac 1n-i\theta|^{-(\beta-\epsilon_1)}.
\]
Recall  $b\in (0,\delta_0 n)$, $n\geq 1$. By (H1) and (H5),
$\|(I-R(e^{-1/n}e^{i\theta}))^{-1}Q(e^{-1/n}e^{i\theta})\|_\cB=O(1)$, for all $|\theta|<b/n$.
This together with the above displayed equation and equation~\eqref{eq-T(z)} yield
\begin{align*}
\Big|\int_{-b/n}^{b/n}(T(e^{-1/n}e^{i\theta})-(1-\lambda(e^{-1/n}e^{i\theta}))^{-1}P) e^{-in\theta}\, d\theta\Big|
&\ll\int_{0}^{b/n}\frac{|\frac 1n-i\theta|^{-(\beta-\epsilon_1)}}{\ell(1/|1/n-i\theta|)}\, d\theta\\
&\ll n^{\beta-\epsilon_1 -1} b^{1-\eps_1-\beta}.
\end{align*}

By  equation~\eqref{eq-ev-Gou} and Lemma~\ref{lemma-asymplambda2},
\[
(1-\lambda(e^{-1/n}e^{i\theta}))^{-1}=\mu(1-e^{-(u+i\theta)\varphi})^{-1}+O(|\frac 1n-i\theta|^{\beta\ggen-\eps_1}).
\]
Recall $q=\max\{ j\geq 0: (j+1)\beta-j>0\}$.
The above displayed equation together with  the argument used in the proof of~\cite[Proposition 9.5]{MT} (which exploits exactly the same assumption on $\mu(\varphi>n)$
stated in item i) of the lemma),
\begin{align*}
\int_{-b/n}^{b/n}(1-\lambda(e^{-1/n}e^{i\theta}))^{-1} e^{-in\theta}\, d\theta
&=d_0 n^{\beta-1}+d_1 n^{2\beta-2}+\ldots+d_q n^{(q+1)(\beta-1)}\\
&+\int_{-b/n}^{b/n}(|\frac 1n-i\theta|^{\beta\ggen-\eps_1}).
\end{align*}

where  $d_0,\ldots, d_q$ are real
constants, depending only on the map $f$ (again, for the precise form of these constants we refer to~\cite[Theorem 9.1]{MT}).

Putting together the last two displayed equations,
\begin{align*}
\int_{-b/n}^{b/n}T(e^{-1/n}e^{i\theta})e^{-in\theta}\, d\theta
&=d_0 n^{\beta-1}+d_1 n^{2\beta-2}+\ldots+d_q n^{(q+1)(\beta-1)}\\
&+O(b^{(1-\beta\ggen+\eps_1)} n^{-(1-\beta\ggen+\eps_1)}+ O(n^{(\beta- \epsilon_1 -1)} b^{1-\eps_1-\beta}).
\end{align*}

Recall $\eps_1>\eps_0$.
Take $b=n^{1/2}$.  So,
\[
\int_{-b/n}^{b/n}T(e^{-1/n}e^{i\theta})e^{-in\theta}\, d\theta
=d_0 n^{\beta-1}+d_1 n^{2\beta-2}+\ldots+d_q n^{(q+1)(\beta-1)}+O( n^{-(1-\beta\ggen+\eps_1)/2}).
\]
To conclude, note that  $|I_{\delta_0}+I_{-\delta_0}+I_A|\ll n^{-1/2}+n^{-(\ggen\beta-\epsilon_1-\delta_3)}\ell(n)^{-1}$.~\end{proof}

At this end we note that Theorem~\ref{cor-mixingrates} follows by the argument used in the proof of Theorem~\ref{cor-equiMTl} (with Theorem~\ref{lemma-FO-Tn-MT} replaced by Theorem~\ref{lemma-HOTn})
and that:

\begin{rmk} \label{rmk-tailf2}
Continuing on Remark~\ref{rmk-tailf}, we note that
if instead of (H4)(ii) we assume (H4)(iii) in the statement of Theorem~\ref{lemma-HOTn} (with the rest of the assumptions unchanged)
then  $E_n=O(n^{-(\beta-1/2)})$ if the  assumption  on $\mu(\varphi>n)$ stated in  Theorem~\ref{lemma-HOTn}, i) holds
and  $E_n=O(n^{-\beta})$ if the  assumption  on $\mu(\varphi>n)$ stated in  Theorem~\ref{lemma-HOTn}, ii) holds.
\end{rmk}

\section{The abstract framework applied to invertible systems (Markov and non Markov examples).}
\label{subsect-classex}

In this section we present some simple, but non trivial, examples to which our abstract framework easily applies.

The examples  considered are far from being the most general (see Remark \ref{rem:general-ex} for details).   Nevertheless, they are fairly representative for both classes of invertible systems: i) preserving and  ii) lacking a Markov structure.
The properties needed to apply the general theory are established in:
a) Section~\ref{sec-setup} in the Markov case; b) Section~\ref{sec-nonMarkov} in the non Markov case.

Let $f:X\to X$ be an invertible map and  $F:Y\to Y$ its first return  map, for some $Y\subset X$.  As already mentioned we use $\vf$ to designate the return time to $Y$ and we set $Y_n=\{x\in Y\;:\;\vf(x)=n\}$.

The strongest restriction in our examples is given by the requirement that there exists a globally smooth stable foliation. In principle, our methods could be applied to more general cases, but this requires a case by case analysis that does not belong to the present paper.
The requirement that $f$ preserves a global smooth foliation means that there exists a smooth map $\bH(x,y)=(H(x,y),y)$, which can be normalized so that $H(x,0)=x$, such that
\[
f\circ \bH(x,y)=(H(f_0(x),g(x,y)), g(x,y))=\bH(f_0(x),g(x,y)),
\]
for some functions $f_0, g$.
Indeed, the fibre through the point $(x,0)$ can be seen as the graph of the function $H(x,\cdot)$ over the $y$ axis and is mapped by $f$ to the fibre thru the point $(f_0(x),0)$. In other words the map $f$ is conjugated to the skew product
\[
\bH^{-1}\circ f\circ \bH(x,y)=(f_0(x),g(x,y)).
\]
Accordingly, in the following, we will consider only maps of the latter form.

\vspace{-2ex}
\paragraph{Example 1: a set of Markov maps.}

Consider $f:[0,1]^2\to [0,1]^2$,
\begin{align} \label{eq-2DLSV}
f(x,y) =(f_0(x), g(x,y)),
\end{align}
where $f_0$ is the map defined in \eqref{eq-LSV}.
We require that $g(x,[0,1])\subset [0,1/2]$ for $x\in [0,1/2)$ and $g(x,[0,1])\subset [1/2, 1]$ for $x\in (1/2, 1]$,
also there exists $\sigma>0$ such that $|\partial_y g|>\sigma$. This implies that $f$ is an invertible map.
Also we assume that $g$ is $C^2$ when restricted to $A_1=(0,1/2)\times [0,1]$ and $A_2=(1/2,1)\times [0,1]$, also we assume $\|\partial_x g\|_{L^\infty}<\infty$.
Setting $R_i=f(A_i)$, it is possible that the closure of $R_1\cup R_2$ is strictly smaller than $[0,1]^2$.
However, $f$  preserves the Markov structure of the map $f_0$. In particular, the preimage of a vertical segment
$\{x\}\times [0,1]$ consists of two vertical segments of the same type. 

Let $Y=(1/2,1]\times [0,1]$ and $F$ be the first return map to such a set. Obviously, it will have the form
$F(x,y)=(F_0(x), G(x,y))$ where $F_0$ is the return map of $f_0$ to $(1/2,1]$. Clearly $\vf(x,y)=\vf_0(x)$ where $\vf_0$ is the return time of the map $f_0$.  We assume that, where defined,
\begin{equation}\label{eq:stable-contraction}
|\partial_y G|\leq \lambda^{-1}<1.
\end{equation}
This implies that the stable foliation consists of the vertical segments.
Also, we require that there exists $K_0>0$ such that
\begin{equation}\label{eq:unstable-cone}
\frac{|\partial_x G|}{|F_0'|}\leq K_0.
\end{equation}
This implies that the cone $\cC=\{(a,b)\in\bR^2\;:\; |b|\leq K|a|\}$ is invariant for $DF$, provided $K\geq (1-\lambda^{-1})^{-1}K_0$.
This readily implies the existence of an unstable foliation and that it is made by curves that are graphs over the $x$ coordinate. 

We note that condition \eqref{eq:stable-contraction} does not require that $f$ is uniformly contracting in the vertical direction.
If, for example, $|\partial_y g|_\infty\leq 1$ and $\sup_{x\in[1/2,1], y}|\partial_y g(x,y)|<1$, 
then one can easily check that \eqref{eq:stable-contraction} and \eqref{eq:unstable-cone} are satisfied.

Indeed, \eqref{eq:stable-contraction} follows trivially. As for condition \eqref{eq:unstable-cone}, we note that setting $f^n=(f_0^n, g_n)$,
we have that $g_{n+1}(x,y)=g(f_0^n(x), g_n(x,y))$. Hence,  for all $(x,y)\in Y$,
\begin{equation}\label{eq:many-der}
\begin{split}
\partial_x g_{n+1}(x,y)&= (\partial_x g)(x_n,y_n)\cdot (f_0^n)'(x)+\partial_y g(x_n,y_n)\cdot\partial_x g_n(x,y)\\
&= (f_0^n)'(x)\sum_{k=0}^n (\partial_x g)(x_k,y_k)\prod_{j=k+1}^n\frac{(\partial_y g)(x_j,y_j)}{f_0'(x_j)},
\end{split}
\end{equation}
where $(x_n,y_n)=f^n(x,y)$. The above displayed equation together with the distortion properties of $F_0=f_0^{\varphi_0}$
and the fact that $|\partial_y g|_\infty\leq 1$, yields
\[
\begin{split}
\frac{|\partial_x G(x,y)|}{|F_0'(x)|}&\leq C \sum_{k=0}^{\vf_0(x)-1} \prod_{j=k+1}^{\vf_0(x)-1}\frac{|(\partial_y g)(x_j,y_j)|}{f_0'(x_j)}\\
&\leq C \sum_{k=0}^{\vf_0(x)-1} (\vf_0(x)-k)^{-1-\beta}\prod_{j=k+1}^{\vf_0(x)-1}|(\partial_y g)(x_k,y_k)|\leq C.
\end{split}
\]

The above argument shows that there exists a large class of systems satisfying hypotheses \eqref{eq:stable-contraction} and \eqref{eq:unstable-cone}.
For such systems the return map $F$ is uniformly hyperbolic, it has a Markov structure with countably many interval of smoothness,
but neither the derivative nor its inverse is, in general, uniformly bounded.
Moreover,  any SRB invariant measure for $f$ (i.e. any invariant measure absolutely continue with respect to Lebesgue once
restricted to the unstable direction) has a marginal in the $x$ direction that is absolutely continuous with respect
to Lebesgue and must be an invariant measure of $f_0$.
Hence, there exists  a unique (up to scaling) $\sigma$-finite,
absolutely continuous (on the unstable direction) invariant measure $\mu$:  finite if $\alpha\leq 1$ and infinite if 
$\alpha\geq 1$ (equivalently, writing  $\beta:=1/\alpha$, $\mu$ is finite if $\beta>1$ and infinite if $\beta\leq 1$).

\vspace{-2ex}
\paragraph{Example 2: a set of non Markov maps.}
In this case we take  $f_{0}$ to be the one dimensional  topologically mixing map described
at the end of Section~\ref{sec-MISyst}. 
We recall that: i) $f_{0}$ agrees with the definition in the previous example in $[0,\frac 12]$; ii) there exists a finite partition of $(1/2, 1]$ into open intervals $I_p$, $p\geq 1$ such that  $f_{0}$ is  $C^2$ and strictly monotone in each $I_p$; iii)
 $|f_{0}'|>2$ in $(1/2, 1]$; iv) $f_0$ is topologically mixing. Define  $f:[0,1]^2\to [0,1]^2$,
\begin{align} \label{eq-2DnM}
f(x,y) =(f_{0}(x), g(x,y)),
\end{align}
where $g\in C^2$ in $(0,1/2)\times [0,1]$ and  in each $I_p\times [0,1]$. As in Example 1, we ask $0<\sigma\leq |\partial_y g|\leq 1$ and that $g$ is such that $f$ is invertible.

Set $Y=(1/2,1]\times [0,1]$ and let $F$ be the first return map to such a set. So, we can write
$F(x,y)=(F_{0}(x), G(x,y))$ where $F_{0}$ is the return map of $f_{0}$ to  $(1/2, 1]$.
It turns out that, to treat this case, conditions of the type~\eqref{eq:stable-contraction} and~\eqref{eq:unstable-cone}, are not sufficient.
Indeed, if the contraction in the stable direction is much slower than the expansion, then it is unclear what is the reasonable result one should expect. To make things simple we ask that the contraction  overbeats the expansion.  We assume that there exists $C>0$ such that, for almost all $(x,y)\in[0,1]^2$,
\begin{equation}\label{eq:extra-cond0}
|\partial_y G(x,y)|\leq C \vf(x,y)^{-1}|F_0'|^{-1}.
\end{equation}
\begin{rmk} We do not claim that condition \eqref{eq:extra-cond0} is optimal, yet it is not very strong either. In particular, note that if we assume the rather strong condition $\|\partial_yg\|_{\infty}\leq \lambda^{-1}<1$, then the above condition reads $\lambda^{-\vf}\leq C\vf^{-2-\beta}$ which is obviously satisfied.\footnote{ Recall that $F_0'\sim \vf^{1+\beta}$.}
\end{rmk}
In the following we use \eqref{eq:extra-cond0} to obtain certain estimates that will be needed in section \ref{sec-nonMarkov}. In some sense these are the properties that are really needed to apply our results, yet we find condition \eqref{eq:extra-cond0} more appealing to state and simpler to check.

Note that, using  the notation of Example 1, $\partial_y G(x,y)=\prod_{k=0}^{\vf(x,y)-1}\partial_yg(x_k,y_k)$. It follows
\[
\partial^2_y G=\sum_{k=0}^{\vf -1}\frac{\partial^2_yg(x_k,y_k)}{\partial_yg(x_k,y_k)}\prod_{j=0}^{k-1}\partial_yg(x_j,y_j)\prod_{j=0}^{\vf-1}\partial_yg(x_k,y_k).
\]
We cannot take much advantage of the first product, so we bound it by one. The second product yields $\partial_yG$ and, using \eqref{eq:extra-cond0}, we have
\begin{equation}\label{eq:extra-cond1}
\|F_0'\partial_y G\|_\infty+\|F_0'\partial^2_y G\|_\infty\leq C
\end{equation}
Moreover, differentiating \eqref{eq:many-der}, we have
\[
\begin{split}
\partial_y\left(\frac{\partial_x G}{F_0'}\right)=&\sum_{k=0}^{\vf-1}\partial_y\partial_x g(x_k,y_k)\partial_y G\prod_{j=k+1}^{\vf-1}f_0'(x_j)^{-1}\\
&+\sum_{k=0}^{\vf-1}\partial_xg(x_k,y_k)\sum_{l=k+1}^{\vf-1}\frac{\partial_y^2 g(x_l,y_l)}{\partial_y g(x_l,y_l)}
\prod_{j=k+1}^{\vf-1}\frac{\partial_y g(x_j,y_j)}{f_0'(x_j)}\cdot \prod_{s=0}^{l-1}\partial_yg(x_s,y_s).
\end{split}
\]
Recall that $\prod_{j=k+1}^{\vf-1}[f_0'(x_j)]^{-1}\leq C (\vf-k)^{-1-\beta}$, while the products of the $\partial_y g$ can be used to recover $\partial_yG$. Using such facts in the above expression, we obtain
\begin{equation}\label{eq:extra-cond}
\left\|F_0'\partial_y\frac{\partial_ xG}{F_{0}'}(x,\cdot)\right\|_{C^0}\leq C.
\end{equation}

In sections \ref{sec-setup} and \ref{sec-nonMarkov} we verify hypothesis (H1--H5) for the above class of examples. By Corollary~\ref{cor-equilG}  we have then the following result:\footnote{We note that 
Corollary~\ref{cor-equilG} is both: a) a weaker version of  Theorem \ref{cor-equiMTl} since it requires (H4)(iii); b) a stronger version of  Theorem \ref{cor-equiMTl} since it covers the whole range $\beta\in(0,1)$; however, we refer to Subsection~\ref{sec-MISyst} for a discussion of previous results in the non-invertible case in the range $\beta>1/2$ and $\beta\leq 1/2$, respectively.
The same comment applies to Remark~\ref{rmk-tailf2} (used in obtaining Proposition~\ref{thm-mixingrates} below) compared with Theorem~\ref{cor-mixingrates}.}
\begin{prop} \label{thm-mixing-nM}
Assume the setting of maps $f$  of the form~\eqref{eq-2DLSV} or~\eqref{eq-2DnM} described in subsection~\ref{subsect-classex}  with $\beta\in (0,1)$.
Let $v,w:[0,1]^2\to\R$ be $C^{1+q}$,  $q\in (\frac{1+\beta}{2+\beta},1]$,  observables supported on $Y$.  Then, there exists a positive constant
$d_0$ (depending only on the map $f$) such that
\[
\lim_{n\to\infty}n^{1-\beta}\int_{[0,1]^2}  v\, w\circ f^n \, d\mu =d_0 \int_{[0,1]^2}v\, d\mu\int_{[0,1]^2} w\, d\mu.
\] 
\end{prop}

As already mentioned in Section~\ref{sec-mix-rates}, mixing  rates for maps of the form~\eqref{eq-LSV} depend heavily on a good expansion of the tail behavior.
A good tail expansion for the invertible map  $f$ of the form~\eqref{eq-2DLSV}  described in subsection~\ref{subsect-classex} follows immediately from the tail expansion for
the map ~\eqref{eq-LSV} (see subsection~\ref{subs-tail}).

On the other hand, we note that at present it is not clear how to obtain the required tail expansion for the return time  associated with non uniformly expanding, non Markov maps\footnote{In~\cite{MT,Terhesiu12} tail expansions of the return time for the one dimensional version of~\eqref{eq-2DLSV} are obtained by exploiting the fact that
the induced  invariant density is $C^\ell$, $\ell>0$. In the one dimensional version of our  example~\eqref{eq-2DnM}, we only
know that  the density is $BV$.}  such as the one
described at the end of Section~\ref{sec-MISyst}. 
 For precisely this reason (although all our hypotheses (H1)--(H5) are shown to hold for the non Markov map
of  the form~\eqref{eq-2DnM} described in subsection~\ref{subsect-classex}),
the next result provides mixing rates just for the case of~\eqref{eq-2DLSV}.
More precisely, by Remark~\ref{rmk-tailf2} we obtain

\begin{prop} \label{thm-mixingrates}
Assume the setting of maps $f$ of the form~\eqref{eq-2DLSV}  described in subsection~\ref{subsect-classex} with $\beta\in (1/2,1)$.
Let $v,w:[0,1]^2\to\R$ be $C^1$ observables supported on $Y$. Set $q=\max\{ j\geq 0: (j+1)\beta-j>0\}$. Then, there exist real constants
$d_0,\ldots, d_q$ (depending only on the map $f$) such that
\[
\int_{[0,1]^2}  v\, w\circ f^n \, d\mu =(d_0 n^{\beta-1}+d_1 n^{2\beta-2}+\ldots+d_q n^{(q+1)(\beta-1)}) \int_{[0,1]^2}v\, d\mu\int_{[0,1]^2} w\, d\mu
+O(n^{-\beta}).
\] 
\end{prop}

\begin{rmk}
 The above mixing rate is optimal and matches the results on mixing rates in~\cite{MT,
Terhesiu12} for maps of the form~\eqref{eq-LSV}. 
\end{rmk}

In the finite measure setting we note that~\cite[Theorem 1.1]{Gouezel04} together with the verification of our hypothesis (H1) and (H5) and (H4)(iii) ( see subsection~\ref{sec-decay}), yields 
\begin{prop} \label{thm-mixing-nM}
Assume the setting of maps $f$ of the form~\eqref{eq-2DLSV} or~\eqref{eq-2DnM} described in subsection~\ref{subsect-classex}  with $\beta>1$.
Let $v,w:[0,1]^2\to\R$ be $C^{1+q}$,  $q\in (\frac{1+\beta}{2+\beta},1]$,  observables supported on $Y$. 
Then
\[
\lim_{n\to\infty} n^{1-\beta}\left|\int_{[0,1]^2}  v\, w\circ f^n \, d\mu -  \int_{[0,1]^2}v\, d\mu\int_{[0,1]^2} w\, d\mu, d\mu \right|=c_0  \int_{[0,1]^2}v\, d\mu\int_{[0,1]^2} w\, d\mu
\] 
where $c_0n^{-\beta+1}=\sum_{k>n}\mu(Y_k)$ for $n$ large enough.
\end{prop}

\begin{rmk}\label{rem:general-ex}
As already mentioned, the above classes of examples are not the most general possible even in the present restrictive case in which a smooth global stable foliation is present. They have been chosen as a reasonable compromise between generality and simplicity of exposition, with the aim of showing how the general theory developed in the next section can be applied to concrete examples. Yet, here is a word on more general possibilities.
First of all the case in which the global is only H\"older should be amenable to similar treatment (only now one does not want to perform the change of variable to obtain a skew product). 
Next, note that there is no reason why the contracting direction should be one dimensional, maps with higher dimensional stable manifolds can be treated in exactly the same manner. Also, one can consider the case in which the expanding direction is higher dimensional. The Markov case would be essentially identical.
In the finite partition non Markov case, one could model the Banach space on higher dimensional bounded variation functions or spaces of generalized variation (see \cite{Keller85, Saussol}). Note however that, as already mentioned, this poses non trivial problems already in the expanding case.
Provided some appropriate technical condition on the image of the partition is satisfied,  the case of (not necessarily Markov) countable partitions can also be treated.
But in the latter case, one would have to use the arguments put forward in \cite{Saussol, liverani11} to prove the relevant spectral properties for the return map.
\end{rmk}

\section{Banach spaces estimates in the Markov case}
\label{sec-setup}

To employ our general framework we must introduce an appropriate Banach space in which to analyze the transfer operators. This is the purpose of the present section.

More precisely, let $f:[0,1]^2\to [0,1]^2$ be  the map~\eqref{eq-2DLSV}  described in subsection~\ref{subsect-classex}.  Let $Y_0=Y=(1/2,1]\times [0,1]$ and let  $\varphi:Y\to\bN$ be the return time to $Y$.
Let $F=f^\varphi$  be the first return map. We show that  $F$ satisfies (H1--H5)  for some appropriate function spaces $\cB, \cB_w$ constructed in analogy with~\cite{DemersLiverani08}. We start by describing the spaces $\cB, \cB_w$.

\subsection{Notation and definitions}\label{subsec-admleaves}

It is convenient to introduce the notation $F^n=(F_0^n,G_n)$, for all $n\in\bN$.
The properties of $F$ can be understood in terms of the map $F_0=f_0^{\varphi_0}$, where 
$\varphi_0$ is the first return time of $f_0$ to $X_0=(1/2,1]$. For $j\geq 1$, set $(x_{j-1},x_j]=X_j=\{\varphi_0=j\}$.  For $j\geq 1$ set $Y_j=  X_j\times [0,1]$. So, $Y_j=\{\varphi=j\}$. For all $j\geq 1$, let 
$ f^jY_j=Y_j'=\{(x,y)\in X_0\times [0,1]\;:\; y\in K_j(x)\} $ for some collection of intervals $K_j(x)\subset [0,1]$. Hence, we can write $F=f^j:Y_j\to Y_j'$. 

 For $n\geq 0$, let  $\cY_n=\{Y_{n,j}\}$ be the corresponding partition of $Y$ associated with $(Y, F^n)$.  Since $F$ is  invertible, we have 
$F^{-n}(\{Y_{n,j}'\})=\{Y_{n,j}\}$.  The map $F^n$ is smooth in the interior of each 
element of the partition $\cY_n$.

\vspace{-1ex}

\paragraph{Admissible leaves:}

We start by introducing a set of {\em admissible leaves} $\Sigma$. Such leaves consists of full vertical segments $W$. A full vertical segment $W(x)$, based at the point $x\in [0,1]$, is given by $\bG_x(t)=(x,t)$, $t\in [0,1]$. The definition of the set of admissible leaves differs with the one in \cite{DemersLiverani08} and allows for a considerable simplification of the arguments. Yet, it is possible only due to the (very special) fact that the map is a skew product.

\vspace{-1ex}
\paragraph{Uniform contraction/expansion, distortion properties:}

Given the simple structure of the stable leaves it is convent to introduce the projection on the second co-ordinate $\pi:[0,1]^2\to [0,1]$ defined by $\pi(x,y)=y$.

By hypothesis \eqref{eq:stable-contraction} we can chose $\lambda>1$ such that:

\begin{itemize}
\item If  $x,y\in W$, $W\in\Sigma$ then $|\pi(F^nx)-\pi(F^ny)|\leq C\lambda^{-n}$. 
\end{itemize}

For any $(x,y)\in Y\in\mathcal{Y}_n$, $|\det(DF^n(x,y))|= (F_0^n)'(x)\cdot \partial_y G_n(y)$.

It is well known that there exists $C>0$ such that, for each $(x,y), (x',y')\in \cY_0$, 
\begin{equation}\label{eq-dist}
\Big|\frac{|F_0'(x)|}{|F_0'(x')|}-1\Big|\leq C |x-x'|.
\end{equation}
In fact, more is true,
\begin{equation}\label{eq-dist-2}
\Big| \frac d{dx}(F_0'(x))^{-1}\Big|\leq C |F_0'(x)^{-1}|.
\end{equation}
\vspace{-1ex}
\paragraph{Test functions:}
In what follows, for $W\in\Sigma$ and $q\leq 1$ we denote by $C^q(W,\bC)$ the Banach space of complex
valued functions on $W$ with  H\"older exponent $q$ and norm
\[
|\phi|_{C^q(W,\bC)}=\sup_{z\in W} |\vf(z)|+\sup_{z,w\in W}\frac{|\vf(z)-\vf(w)|}{|z-w|^q}.
\]
Note that $C^q(W(x),\bC)$ is naturally isomorphic to $C^q([0,1],\bC)$ via the identification of the domain given by $t\to (x,t)$. 
In the following we will use implicitly such an identification, in particular for $\phi\in C^q([0,1],\bC)$ we still call $\phi$ the corresponding function in $C^q(W(x),\bC)$ and we write
\[
\int_{W(x)} h\phi\, dm=\int_{0}^{1} h(x,t)\phi(t) dt.
\]
\begin{rmk}
Note that we use $m$ both for the one dimensional and two dimensional Lebesgue measure. Also, in the following we will often suppress $dm$ as this does not create any confusion.
\end{rmk}
\vspace{-1ex}
\paragraph{Definition of the norms:}
Given $h\in C^1(Y,\C)$, define the \emph{weak norm} by
\begin{equation}\label{eq-weaknorm}
\|h\|_{\cB_w}:=\sup_{W\in\Sigma}\;\sup_{|\phi|_{ C^1(W,\bC)}\leq 1 }\int_W h\phi\, dm.
\end{equation}

Given $q\in [0,1)$ we define the \emph{strong stable norm} by

\begin{equation}\label{eq-strongnormst}
\|h\|_s:=\sup_{W\in\Sigma}\;\sup_{|\phi|_{ C^q(W,\bC)}\leq 1 }\int_W h\phi\, dm.
\end{equation}

Finally we define the \emph{strong unstable norm} by
\begin{equation}\label{eq-strongnormunst}
\|h\|_u:=C^{-1}\sup_{x,y\in [1/2,1]}
\sup_{|\phi|_{C^1}\leq 1}\frac{1}{|x-y|}\left|\int_{W(x)} h\phi\, dm
-\int_{W(y)} h\phi\, dm\right|.
\end{equation} 
Finally, the \emph{strong norm} is defined by $\|h\|_{\cB}=\|h\|_s+\|h\|_u$.

\paragraph{Definition of the Banach spaces:}
We will see briefly that $\|h\|_{\cB_w}+\|h\|_{\cB}\leq C\|h\|_{C^1}$. We then define $\cB$ to be the completion of $C^1$ in the strong norm and $\cB_w$ to be the completion in the weak norm.

The spaces $\cB$ and $\cB_w$ defined above are simplified versions of functional space defined
in~\cite{DemersLiverani08} (adapted to the setting of~\eqref{eq-2DLSV}).
The main difference in the present setting is the simpler definition of admissible leaves and the absence of a control on short leaves. The latter is necessary and possible since the discontinuities do no satisfy any transversality condition while, instead, they enjoy some form of Markov structure.  

\subsection{Embedding properties: verifying (H1)(i)}

The next result shows that (H1)(i) holds for  $\alpha= \gamma=1$ and $\cB, \cB_w$ as described above.

\begin{lemma}\label{lemma-embed} For all $q\in (0,1)$ in definition \eqref{eq-strongnormst} we have\footnote{ Here the inclusion is meant to signify a continuous embedding of Banach spaces.}
\[
C^1\subset \cB\subset \cB_w\subset(C^1)'.
\]
Moreover, the unit ball of $\cB$ is relatively compact in $\cB_w$.
\end{lemma}
\begin{proof}
By the definition of the norms it follows that $\|\cdot\|_{\cB_w}\leq \|\cdot\|_s\leq\|\cdot\|_{\cB}$, from this the inclusion $\cB\subset \cB_w$ follows. 
For each $h\in C^1$ we have
\[ 
\left|\int_{W(x)} h\phi-\int_{W(y)} h\phi\right|=\left|\int_0^1 h(x,t)\phi(t) dy-\int_0^1 h(y,t)\phi(t)\right|\leq C\|h\|_{C^1}\|\phi\|_{C^0}|x-y|.
\]
The above implies $\|h\|_u\leq C \|h\|_{C^1}$. Thus $C^1\subset \cB$.

The other inclusion is an immediate consequence of Proposition \ref{prop-embed}, an analogue of~\cite[Lemma 3.3]{DemersLiverani08}.
The injectivity follows from the injectivity of the of the standard inclusion of $C^1$ in $(C^1)'$.

Finally, we want to show that the unit ball $B_1$ of $\cB$ has compact closure. To this end it suffices to show that it is totally bounded, i.e. for each $\ve>0$, it can be covered by finitely many $\ve$-balls in the $\cB_w$ norm. 

For each $\ve>0$ let $\cN_\ve=\{x_i=i\ve\}_{i=1}^{\ve^{-1}}\subset [0,1]$. Hence, for each admissible leaf $W$ and test function $\phi$ there exists $x_i\in\cN_\ve$ such that
\[
\left|\int_Wh\phi-\int_{W(x_i)}h\phi \right|\leq \ve\|h\|_u.
\]
On the other hand, by Ascoli-Arzela, for  each $\ve$ there exist finitely many $\{\phi_i\}_{i=1}^{M_\ve}\subset C^1([0,1],\bC)$ which is $C^q$ $\ve$-dense in the unit ball of $C^1$.
Accordingly, for each $\phi$ such that $\|\phi\|_{C^1}\leq 1$ we have that there exists $\phi_j$ such that
\begin{equation}\label{eq:compact-short}
\left|\int_Wh\phi-\int_{W(x_i)}h\phi_j \right|\leq C \ve\|h\|\,\|\phi\|_{C^1}.
\end{equation}
Let $K_\ve:\cB_w\to\bR^{N_\ve M_\ve}$ be defined by $[K_\ve(h)]_{i,j}=\int_{W(x_i)}h\phi_j $. Clearly $K_\ve$ is a continuous map.
Since the image of the unit ball $B_1$ under $K_\ve$ is contained in $\{a\in \bR^{N_\ve M_\ve}\;:\; |a_{ij}|\le 1\}$, it has a compact closure hence there exists finitely many $a^k\in \bR^{N_\ve M_\ve}$ such that the sets
\[
U_{k,\ve}=\left\{h\in\cB\;:\; \left|\int_{W(x_i)}h\phi_j-a^k_{i,j}\right|\leq \ve, \,\forall i,j\right\}
\]
cover $B_1$. To conclude note that if $h_1,h_2\in U_{k,\ve}$, then, by \eqref{eq:compact-short}, there exists $i,j$ such that
\[
\begin{split}
\left|\int_W(h_1-h_2)\phi\right|&\leq \left|\int_{W(x_{i})}(h_1-h_2)\phi_{j}\right|+C\ve\|h_1-h_2\|_{\cB}\\
&\leq \left|\int_{W(x_{i})}h_1\phi_{j}-a^k_{i,j}\right|+\left|\int_{W(x_{i})}h_2\phi_{j}-a^k_{i,j}\right|+2C\ve\\
&\leq 2(C+1)\ve.
\end{split}
\]
This means that each $U_{k,\ve}$ is contained in a $2(C+1)\ve$-ball in the $\cB_w$ norm, hence the claim.
\end{proof}
\begin{prop}\label{prop-embed}Let $I$ be an interval, $I\subset [0,1]$ and set $E=I\times [0,1]$. Then for all $\phi\in C^1$
and for all $h\in\cB_w$, we have $\Id_Eh\in \cB_w$ and
\[
|\langle\Id_E h\phi\rangle|\leq  \|h\|_{\cB_w} \|\phi\|_{C^1} m(E).
\]
\end{prop}
\begin{proof}
Let us start by considering $h\in C^1$ and let $g\in L^\infty$ such that $\partial_y g=0$. By Fubini theorem and the fact that the vertical segments are admissible leaves, for each $\phi\in C^1$ we have
\[
\left|\int  h g\phi\, \right|\leq\int dx |g(x)|\left|\int_0^1 dy h(x,y)\phi(x,y)\right|\leq \|h\|_{\cB_w}\|\phi\|_{C^1} \|g\|_{L^1}.
\]
In particular, this implies that, choosing $\{g_n\}\subset C^1$ such that $g_n\to \Id_E$ in $L^1$, $\|\Id_E h-g_nh\|_{\infty}\to 0$, i.e. $\Id_E h\in\cB_w$. Moreover, 
\[
|\langle\Id_E h\phi\rangle|\leq \|h\|_{\cB_w}\|\phi\|_{C^1} m(E).
\]
In other words, if we view the multiplications by $\Id_E$ as an operator on $C^1\subset\cB_w$, then $\|\Id_E h\|_{\cB_w}\leq \|h\|_{\cB_w}$, that is $\Id_E$ is a bounded operator. It can then be extended uniquely to a bounded operator on $\cB_w$ by the same norm (and name).
\end{proof}

\subsection{ Transfer operator: definition}
If $h\in L^1$, then
$R:L^1\to L^1$ acts on $h$ by
\[
\int Rh\cdot v=\int h\cdot v\circ F,\quad v\in L^\infty.
\]
By a change of variable we have
\begin{equation}\label{eq:transfer-L1}
Rh=\Id_{F(Y)} h\circ F^{-1}\det(DF^{-1}).
\end{equation}
Note that, in general, $R C^1\not\subset C^1$, so it is not obvious that the operator $R$ has any chance of being well defined in $\cB$.
The next Lemma addresses this problem.

\begin{lemma}\label{lem:R-well-def}
With the above definition, $R(C^1)\subset \cB$.
\end{lemma}
\begin{proof}
Using the notation introduced at the beginning of section \ref{subsec-admleaves}, we have $F(Y)=\cup_j Y_j'$. Moreover, both $F^{-1}$ and $\det(DF^{-1})$ are $C^1$ on each $\overline{Y_j'}$. For each $j\in\bN$ note that $Y_j'$ consists of an horizontal strip bounded by the curves $\gamma_1(x)=G(g_j(x),1)$, $\gamma_0(x)=G(g_j(x),0)$ where $g_j:(\frac 12,1]\to(\frac 12,1]$ is the inverse branch of $F_0$ corresponding to the return time $j$. Remark that, by equation \eqref{eq:unstable-cone}, it follows that $|\gamma_i'|_\infty\leq K_0$, for $i\in\{0,1\}$. We can then consider a sequence of $\bar\psi_{n}\in C_0^1(\bR, [0,1])$ that converges monotonically to $\Id_{[0,1]}$ and define $\tilde\psi_n(x,y)=\bar\psi_n(y)$. Next, we define the function $\psi_{n}=\tilde\psi_n\circ F^{-1}\cdot \Id_{F(Y)}$. Note that $\psi_{n}$ is smooth and converges monotonically to $\Id_{F(Y)}$. We can then define 
\[
H_{n}= \psi_{n}h\circ F^{-1}\det(DF^{-1})\in C^1.
\]
Consider an admissible leaf $W=W(x)$ and a test function $\phi$. Note that $F^{-1}W=\cup_j W_j$ where $W_j=F^{-1}W\cap Y_j =W(g_j(x))$ are vertical leaves. Thus, given $h\in C^1$,\footnote{ Since $W_j\subset Y$,  $F(W_j)\subset F(Y)$. Thus, $\Id_{F(Y)}\circ F$, restricted on $W_j$, equals one.}
\[
\begin{split}
\left|\int_W[Rh-H_{n}]\phi\right|&\leq \sum_j\int_{W_j} |h|\, |F_0'|^{-1}|\phi\circ F|\,|1-\tilde\psi_{n}|\\
&\leq \|h\|_\infty\|\phi\|_\infty C\sum_{j}|F_0'(g_j(x))|^{-1}\int_0^1 |1-\bar\psi_{n}(t)|dt
\end{split}
\]
since $\det(DF^{-1})\circ F=\det(DF)^{-1}=(F_0'\cdot \partial_y G)^{-1}$ and where we have used \eqref{eq-dist} in the second line. Also we have used, and will use in the following, a harmless abuse of notation insofar we write $\phi\circ F^n$ to mean $\phi\circ\pi\circ F^n$. Since the sun is convergent and the integral converge to zero, it follows that the right hand side can be made arbitrarily small by taking $n$ large enough. It follows that $H_{n}$ converges to $Rh$ in $\cB_w$.

The above computation also shows that $\lim_{n\to\infty}\|Rh-H_{n}\|_s=0$. Thus, it remains to check the unstable norm. Let $x,z\in[1/2,1]$. Let $x_j=g_j(x)$, $(x_j,0)\in Y_j$, and $z_j=g_j(z)$, $(z_j,0)\in Y_j$. Then, for each $\phi$, $\|\phi\|_{C^1}\leq 1$, we have
\[
\begin{split}
&\left|\int_{W(x)}[Rh-H_{n}]\phi-\int_{W(z)}[Rh-H_{n}]\phi\right|\\
\leq&\sum_{j}\int_0^1\left| h(x_j,t)F_0'(x_j)^{-1}\phi\circ F(x_j,t)-h(z_j,t)F_0'(z_j)^{-1}\phi\circ F(z_j,t)\right|\,\left|1-\bar\psi_n(t)\right|.
\end{split}
\]
Since, by hypothesis and equations \eqref{eq:unstable-cone}, \eqref{eq-dist-2}, $|\frac d{dx}\left[h(\cdot, t)(F_0)'(\cdot)^{-1}\phi\circ F(\cdot,t)\right]|\leq C$ for some fixed $C>0$, we have
\[
\left|\int_{W(x)}[Rh-H_{n}]\phi-\int_{W(x)}[Rh-H_{n}]\phi\right|\leq C\sum_j|x_j-z_j|\int_0^1\left|1-\bar\psi_n(t)\right|.
\]
Finally, by equation \eqref{eq-dist}, we have $|x_j-z_j|\leq C |F_0'(x_j)|^{-1}|x-z|$ and again we can conclude as above.
\end{proof}

\subsection{ Lasota--Yorke inequality and compactness: verifying (H5)(i) and (H1)(ii).}

The next Lemma is the basic result on which all the theory rests.

\begin{prop} (Lasota--Yorke inequality.) \label{prop-LSinequality} For each $z\in\overline\bD$, $n\in\bN$ and $h\in C^1(Y,\bC)$ we have  
\[
\begin{split}
&\|R(z)^n h\|_{\cB_w}\leq C|z|^n\|h\|_{\cB_w}\\
&\|R(z)^n h\|_{\cB} \leq  \lambda^{-nq}|z|^n \|h\|_{\cB}+C|z|^n \|h\|_{\cB_w}.
\end{split}
\]
\end{prop}
\begin{proof}
Setting $\vf_n=\sum_{k=0}^{n-1}\vf\circ F^k$ we have that $R(z)^n h=R^n(z^{\vf_n} h)$. Remark that $\vf_n$ is constant on the elements of $Y_{n,j}$ of $\cY_n$, moreover $\vf_n\geq n$, hence $|z^{\vf_n}|\leq |z|^n$.

Given $W\in\Sigma$, with base point $x$, we have $F^{-n}(W)=\cup_{j\in\bN} W_j$ where $\cW=\{W_j\}_{j\in\bN}\subset \Sigma$ is the collections of the maximal connected components.
Note that each $Y_{n,j}\in\cY_n$ contains precisely one $W_j$. 
Then, for $|\phi|_{C^1(W,\bC)}\leq 1$, we have
\[
\int_W (R(z)^n h)\phi\, dm=\int_Wz^{\vf_n}\Id_{F^n(Y)}h\circ F^{-n}\det(DF^{-n})\phi.
\]
By the invertibility of $F$, the connected components of $W\cap F^n(Y)$ are exactly $\{F^nW'\}_{W'\in \cW}$. Notice that $F^n(x,y)=(F_0^n(x),G_n(x,y))$, while $F^{-n}(x,y)$ has the more general form $(A(x,y),B(x,y))$. Yet, the function $A$ depends on $y$ only in a limited manner: $A(x,y)=g_j(x)$ for all $(x,y)\in F^n(Y_{n,j})$. Also, it is convenient to call $B_j$ the function $B$ restricted to $F^n(Y_{n,j})$. If $(x_j,0)\in W_j\in \cW$, then $G_n(x_j,t)$ provides a parametrization for the little segment $F^nW_j\subset W$. In addition,  for $(x, y)\in F^nW_j$, $F^{-n}(x,y)=(x_j, H_{n,j}(y))$ with $x_j=g_j(x)$ and $H_{n,j}(y)=B_j(x,y)$. We can then write
\[
\begin{split}
\int_{F^nW_j}h\circ F^{-n}\det(DF^{-n})\phi=&\int_{G_n(x_j,0)}^{G_n(x_j,1)} \frac{h(x_j, H_{n,j}(t))}{(F_0^n)'(x_j)\partial G_n(x_j, H_{n,j}(t))}\phi(t) dt\\
=&\int_0^1 \frac{h(x_j, t)}{(F_0^n)'(x_j)}\phi(G_n(x_j,t)) dt.
\end{split}
\]
By the above computation we have
\[
\left|\int_W (R(z)^n h)\phi\, dm\right|\leq 
\sum_{W_j\in\mathcal{W}} \left|\int_{W_j} z^{\vf_n} h [(F_0^n)']^{-1}\phi\circ  F^n\,dm\right|\leq \sum_{Y_{n,j}\in\cY_n}\|h\|_{\cB_w}\left|\frac{\phi\circ  F^n}{(F_0^n)'}\right|_{C^1(W_j)}\hskip-.5cm |z|^n .
\]
Note that \eqref{eq-dist-2} and \eqref{eq-dist} imply 
\[
\|[(F_0^n)']^{-1}\phi\circ  F^n\|_{C^0(W_j)}\leq C \sup_{x\in W_j}[(F_0^n)']^{-1}\leq 2C m(Y_{n,j}).
\]
W.r.t. the $C^1$ norm, we have
\[
\|[(F_0^n)']^{-1}\phi\circ  F^n\|_{C^1(W_j)}\leq \|[(F_0^n)']^{-1}\|_{C^1}\|\phi\circ  F^n\|_{C^0}+\|[(F_0^n)']^{-1}\|_{C^0}\|\phi\circ  F^n\|_{C^1}.
\]
Recall that $\phi\circ F^n$ stands for $\phi\circ\pi\circ F^n$, thus\footnote{ In the following we will implicitly use standard computations of this type.} 
\[
\begin{split}
\|\phi\circ  F^n\|_{C^1}&\leq |\phi|_\infty+\sup_t|\frac{d}{dt} \phi\circ\pi\circ F^n(x_j,t)|\leq |\phi|_\infty+|\phi'\circ \pi\circ F^n\cdot \partial_yG_n(x_j,\cdot)|_\infty\\
&\leq |\phi|_\infty+|\phi '|_\infty\lambda^{-n}\leq C\|\phi\|_{C^1}.
\end{split}
\]
Thus
\begin{equation}\label{eq:weak-ly-test}
\|[(F_0^n)']^{-1}\phi\circ  F^n\|_{C^1(W_j)}\leq C m(Y_{n,j}).
\end{equation}
Equation \eqref{eq:weak-ly-test} allows to estimate the weak norm as follows
\begin{equation}\label{eq:weak-ly}
\left|\int_W (R(z)^n h)\phi\, dm\right|\leq C\|h\|_{\cB_w}|z|^n\sum_j m(Y_{n,j})\leq C\|h\|_{\cB_w}|z|^n.
\end{equation}
The first inequality of the proposition follows. Let us discuss the strong stable norm. Given  $|\phi|_{C^q(W,\bC)}\leq 1$, we have
\[
\begin{split}
\left|\int_W (R(z)^n h)\phi\, dm\right|&\leq \sum_{W_j\in\cW} \left|\int_{W_j} h [(F^n)']^{-1}\phi\circ  F^n\,dm\right|\,|z|^n\\
&\leq  \sum_{W_j\in\cW} \left|\int_{W_j} h \hat \phi_j\,dm\right|\,|z|^n+\left|\int_{W_j} h [(F^n)']^{-1}\bar \phi_j\,dm\right|\,|z|^n,
\end{split}
\]
where $\bar \phi_j=|W_j|^{-1}\int_{W_j}\phi\circ  F^n$ and $\hat \phi_j=[(F^n)']^{-1}(\phi\circ  F^n-\bar\phi_j)$.
Let $J_{W_j}F^{n}$ be the stable derivative on the fibre $W_j$ (recall that $|J_{W_j}F^{n}|\leq \lambda^{-n}$), then
\[
\begin{split}
&|\phi\circ  F^n-\bar\phi_j|\leq |J_{W_j}F^{n}|_\infty^{q}|W_j|^q\leq C|F^n(W_j)|^q\\
&\sup_{x,y\in W_j}\frac{|\hat\phi_j(x)-\hat\phi_j(y)|}{\|x-y\|^q}\leq C |(F^n)'|_{L^\infty(W_j)}^{-1}|J_{W_j}F^{n}|_\infty^{q}
\leq  \frac{C|F^n(W_j)|^q}{ |(F^n)'|_{L^\infty(W_j)}|W_j|^q}.
\end{split}
\]
Hence, 
\[
\|\hat \phi_j\|_{C^q(W_j,\bC)}\leq C \frac{|F^n(W_j)|^q}{ |(F^n)'|_{L^\infty(W_j)}}\leq C m(Y_{n,j})\lambda^{-nq}.
\]
The above bound yields,
\[
\left|\int_W (R(z)^n h)\phi\, dm\right|\leq C\|h\|_s\lambda^{-nq}\,|z|^n +C\|h\|_{\cB_w}\,|z|^n.
\]
We are left with the strong unstable norm. Let $\|\phi\|_{C^1}\leq 1$, $x,y\in [1/2,1]$.
Let $\cW(x)$ be the set of pre images of $W(x)$ under $F^n$ and the same for $\cW(y)$. Note that to each element of $W_j(x)$ it corresponds a unique element $W_j(y)$ that belongs to the same set $Y_{n,j}\in\cY_n$. Let $\xi_j,\eta_j\in [0,1]$ be such that $W_j(x)=W(\xi_j)$ and $W_j(y)=W(\eta_j)$.
By the usual distortion estimates we have $|\xi_j-\eta_j|\leq C |(F_0^n)'(\xi_j)|^{-1}|x-y|$. We introduce the function $\Phi_{n}=|(F_0^n)'|^{-1}(\eta_j)\cdot\phi\circ  F^n|_{W(\eta_j)}$, and write
\begin{equation}\label{eq:strong-ly}
\begin{split}
&\left|\int_{W(x)}\hskip-.4cm R(z)^nh\phi\, dm-\int_{W(y)}\hskip-.4cm R(z)^nh\phi\, dm\right|\leq\sum_j\left|\int_{W(\xi_j)}\hskip-.4cm h \Phi_{n}\, dm-\int_{W(\eta_j)}\hskip-.4cm h\Phi_{n}\, dm\right|\,|z|^n\\
&\quad+\sum_j\left|\int_{W(\xi_j)} h \left[\,|(F_0^n)'(\xi_j)|^{-1}-|(F_0^n)'(\eta_j)|^{-1}\right]\phi\circ  F^n\, dm\right|\,|z|^n\\
&\leq \sum_j \lambda^{-n}|z|^n|x-y|\left[\|h\|_u+C\|h\|_s\right]|(F_0^n)'|_{L^\infty(W(x_j))}^{-1}\leq C\lambda^{-n} |z|^n|x-y| \|h\|,
\end{split}
\end{equation}
where we have used that $\phi\circ  F^n|_{W(\eta_j)}=\phi\circ  F^n|_{W(\xi_j)}$.
The Lemma follows then by iterating the above formula.
\end{proof}

The above Proposition, together with Lemma \ref{lem:R-well-def}, readily implies that $R(z)\in L(\cB,\cB)$, i.e. Hypothesis (H1)(ii) holds true. Note that Proposition \ref{prop-LSinequality} alone would not suffice, indeed the fact that a function has a bounded norm does not imply that it belongs to $\cB$: for this, it is necessary to prove that it can be approximate by $C^1$ functions in the topology of the Banach space.

The proof of Lemma  \ref{lem:R-well-def} holds essentially unchanged also for the operator $R(z)$, thus $R(z)\in L(\cB,\cB)$. We can then extend, by density, the statement of Proposition  \ref{prop-LSinequality} to all $h\in \cB$, whereby proving hypothesis (H5)(i).

\subsection{Verifying  (H1)(iii) and (H5)(ii)}

The following Lemma is an immediate consequence of Lemma \ref{prop-LSinequality} and the compact embedding stated in Lemma \ref{lemma-embed} (e.g. see \cite{hen}).
\begin{lemma}\label{lem:compactenss}
For each $z\in\overline\bD$ the operator $R(z)$ is quasi-compact with spectral radius bounded by $|z|$ and essential spectral radius bounded by $|z|\lambda^{-q }$.
\end{lemma}

Note that $1$ belongs to the spectrum of $R$ (since the composition with $F$ is the dual operator to $R$ and $1\circ F=1$). By the spectral decomposition of $R$ it follows that $\frac 1n\sum_{i=0}^{n-1} R^n$ converges (in uniform topology) to the eigenprojector $\Pi$ associated to the eigenvalue $1$. Let $\mu=\Pi 1$. 
\begin{rmk}\label{rmk:mu}
Note that, by construction, $\mu=\mu_0\times m$ where $\mu_0$ is the unique SRB measure of $f_0$ and $m$ the Lebesgue measure.
\end{rmk}
The next step is the characterization of the peripheral spectrum.
\begin{lemma}\label{lem:peripheral}
Let $\nu\in\sigma(R(z))$ with $|\nu|=1$. Then any associated eigenvector $h$ is a complex measure. Moreover, such measures are all absolutely continuous with respect to $\mu$ and have bounded Radon-Nikodym derivative.
\end{lemma}
\begin{proof}
Note that it must be $|z|=1$ since the spectral radius of $R(z)$ is smaller or equal to $|z|$. Next, let $h$ be an eigenvector with eigenvalue $\nu$, then $h\in \cB_w\subset (C^1)'$ and, for each $\phi\in C^1$, we have
\[
|h(\phi)|=\left|\nu^{-n} h(z^{\varphi_n} \phi\circ F^n)\right|\leq \sum_j |h(z^j\Id_{Y_{n,j}} \phi\circ F^n)|\leq \|h\|_{\cB_w}\|\phi\circ F^n\|_{C^1}
\]
where we have used Proposition \ref{prop-embed}. Since $\lim_{n\to\infty}\|\phi\circ F^n\|_{C^1}=\|\phi\|_{C^0}$ it follows that $h\in (C^0)'$, i.e. it is a measure. Since, by Lemma \ref{lem:compactenss}, the projector $\Pi_\nu(z)$ on the eigenspace associated to $\nu$ can be obtained as
\[
\lim_{n\to\infty}\frac 1n\sum_{i=0}^{n-1}\nu^{-i}R(z)^i
\]
and since the range of $\Pi_\nu(z)$ is finite dimensional, there must exists $\psi\in C^1$ such that $h=\Pi_\nu(z)\psi$. Then, for each $\phi\in C^1(Y,\bR_+)$, 
\[
\left|\int \phi \frac 1n\sum_{i=0}^{n-1}\nu^{-i}R(z)^i\psi\right|=|\psi|_\infty\int \frac 1n\sum_{i=0}^{n-1}R^i 1 \phi
\]
which, taking the limit for $n\to\infty$, implies $|h(\phi)|\leq |\psi|_\infty \mu(\phi)$ and the Lemma.
\end{proof}

Now suppose that $R h=e^{i\theta}h$. Then, by the above Lemma, there exists $v\in L^\infty(\mu)$ such that $h=v\mu$. Hence
\[
\mu (v\phi)=h(\phi)=e^{-i\theta}h(\phi\circ F)=e^{-i\theta}\mu( v\phi\circ F)=e^{-i\theta}\mu(\phi v\circ F^{-1})
\]
implies $v=e^{i\theta}v\circ F\,$ $\mu$-almost surely. By similar arguments, if $z=e^{i\theta}$ and $R(z)h=R(e^{i\theta\vf}h)=h$, then there exists $v\in L^\infty(\mu)$ such that $ve^{i\theta \varphi}=v\circ F\,$ $\mu$-almost surely.

\begin{prop} Hypotheses (H1)(iv) and (H5)(ii) hold true. 
\end{prop}
\begin{proof} As the proof of the two hypotheses is essentially the same, we limit ourselves to the proof of (H5)(ii). Let $v:Y\to\C$ be a (non identically zero) measurable solution to the equation  $v\circ F=e^{i\theta\varphi}v$ a.e. on $Y$, with
$\theta\in (0, 2\pi)$. By Lusin's theorem, $v$ can be approximated in $L^1(\mu)$ by a $C^0$ function, which in turn can be approximated by a
$C^\infty$ function. Hence, there exists a sequence $\xi_n$ of $C^1$ functions  such that $|\xi_n-v|_{L^1(\mu)}\to 0$, as $n\to\infty$. So,
we can write 
\[
v=\xi_n+\rho_n,
\] where $|\rho_n|_{L^1(\mu)}\to 0$, as $n\to\infty$.

Starting from $v=e^{-i\theta\varphi} v\circ F$ and iterating forward $m$ times (for some $m$ large enough to be specified later),
\begin{align*}
v=e^{-i\theta\sum_{j=0}^{m-1}\varphi_0\circ F_0^j}( \xi_n\circ F^m+\rho_n\circ F^m).
\end{align*}

Clearly,
\begin{align}\label{eq-rhon}
|e^{-i\theta\sum_{j=0}^{m-1}\varphi_0\circ F_0^j}\rho_n\circ F^m|_{L^1(\mu)}= |\rho_n\circ F^m|_{L^1(\mu)}= |\rho_n|_{L^1(\mu)}\to 0,
\end{align}
as $n\to\infty$.

Next, put $A_{n,m}:=e^{-i\theta\sum_{j=0}^{m-1}\varphi_0\circ F_0^j} \xi_n\circ F^m$ and note that for all $n$ and $m$
\[
|\partial_y A_{n,m}|\leq |\partial_y \xi_n|_{\infty}|\partial_y F^m|.
\]

By condition~\eqref{eq:stable-contraction}, there exists $0<\tau<1$ such that $|\partial_y F|=\tau$. Hence, for any $\ve>0$ and any
$n\in\bN$, there exists $m\in\bN$ such that
\[
|\partial_y A_{n,m}|_\infty<\varepsilon.
\]
It is then convenient to use $\bE_\mu$ for the expectation with respect to $\mu$ and $\bE_\mu(\cdot\;|\; x)$ for the conditional expectation with respect to the $\sigma$-algebra generated by the set of admissible leaves.
As a consequence, $|A_{n,m}(x,y)-\bE_\mu(A_{n,m}\;|\;x)|\leq \varepsilon$.
For arbitrary $\psi\in L^\infty(\mu)$, we can then write

\begin{equation}\label{eq-psiA}
\begin{split}
\bE(\psi v)&=\bE_\mu(\psi A_{n,m})+O(\ve)=\bE_\mu(\psi \bE_\mu(A_{n,m}\;|\;x))+O(\varepsilon)\\
&=\bE_\mu( \bE_\mu(\psi\;|\;x)A_{n,m})+O(\varepsilon)=\bE_\mu(\psi \bE_\mu( v\;|\;x))+O(\ve).
\end{split}
\end{equation}
By the arbitrariness of $\ve$ and $\psi$ it follows $v=\bE_\mu( v\;|\;x)$. But this implies that $v\circ F_0=v\circ F=e^{i\theta\vf_0}v$, but this has only the trivial solution $v=0$ (see~\cite[Theorem 3.1]{AaronsonDenker01}).
\end{proof}

\subsection{Verifying (H4): bounds for $\|R_n\|_{\cB}$.}

The next result shows that the strongest form of (H4), that is (H4)(iii), holds.

\begin{lemma}\label{lem:asymptotic}
For each $n\in\bN$ we have the bound
\[
\|R_n \|_{\cB}\leq C n^{-\beta-1}.
\]
\end{lemma}
\begin{proof}
Note that $\vf$ is constant on $\cY_1$, hence there exists $j_n$ such that $\vf|_{Y_{j_n}}=n$. Thus for each $\|\phi\|_{C^q}\leq 1$ and $W\in\Sigma$,
\[
\left|\int_W (R_n h)\phi\, dm\right|\leq 
 \left|\int_{W_{j_n}}  h |F_0'|^{-1}\phi\circ  F\,dm\right|\leq \|h\|_s\left|F_0'\right|^{-1}_{C^q(W_{j_n})} \leq C n^{-\beta-1}.
\]
Next, let $x,y\in[0,1]$, $|x-y|\leq\ve_0$, and $\phi\in C^1$. Setting $\Phi_{n}=|F_0'|^{-1}(\eta_j)\cdot\phi\circ  F|_{W(\eta_j)}$,
\[
\begin{split}
&\left|\int_{W(x)} R_nh\phi\, dm-\int_{W(y)} R_nh\phi\, dm\right|\leq\left|\int_{W(\xi_{j_n})} h \Phi_{n}\, dm-\int_{W(\eta_{j_n})} h\Phi_{n}\, dm\right|\,\\
&\quad+\left|\int_{W(\xi_{j_n})} h \left[\,|F_0'(\xi_{j_n})|^{-1}-|F_0'(\eta_{j_n})|^{-1}\right]\phi\circ  F\, dm\right|\,\\
&\leq  |x-y|\left[\lambda^{-1}\|h\|_u+C\|h\|_s\right]|(f_0^n)'(\xi_{j_n})|^{-1}\leq C n^{-\beta-1}|x-y| \|h\|_{\cB},
\end{split} 
\]
where, in the last line, we have used \eqref{eq-dist-2}.
\end{proof}

\subsection{Verifying (H2)}
We note that the connected components of $\varphi^{-1}(n)$ satisfy the assumption on the set $E$ in the statement of
Proposition~\ref{prop-embed}. Hence,
\[
\left|\int_E h\, dm\right|\leq  \|h\|_{\cB_w} m(E),
\]
and (H2) (with $\ggen=1$) follows by Remark \ref{rmk:mu}, since $\frac {dm}{d\mu_0}$ is known to be bounded on $Y$.

\subsection{Verifying (H3)}
\label{subs-tail}
To conclude we must verify (H3). Again the strategy is to reduce to the one dimensional map $F_0$. Indeed, consider $\psi$ such that $\psi(x,y)=\bE_\mu(\psi\;|\;x)$, then
\[
\mu(\psi\circ F_0)=\mu(\psi\circ F)=\mu(\psi)
\]
this implies that the marginal of $\mu$ is the invariant measure $\mu_0$ of the map $F_0$.
Since $\vf$ does not depend on $y$, (H3) holds for $F$ since it holds for $F_0$ (see, for instance, \cite{LiveraniSaussolVaienti99}).

The argument above  together with the the tail expansion of  $\mu_0(\varphi_0>n)$ ( associated with $f_0$) obtained in~\cite{MT, Terhesiu12}
shows that the conditions on the tail behavior $\mu(\varphi>n)$ (associated with
the $f$ ) stated in  Theorem~\ref{lemma-HOTn}, (i)-(ii) are satisfied.

\section{Banach spaces estimates in the non Markov setting}
\label{sec-nonMarkov}

Let $f:[0,1]^2\to [0,1]^2$ be  the non Markov map~\eqref{eq-2DnM}  introduced in subsection~\ref{subsect-classex}.  Let $Y_0=Y=(1/2,1]\times [0,1]$ and let  $\varphi:Y\to\Z^{+}$ be the return time to $Y$.
Let $F=f^{\varphi}$  be the first return map and write $F^n=(F_{0}^n,G_n)$, for all $n\in\bN$ and 
$F_{0}=f_{0}^{\varphi_{0}}$, where 
$\varphi_{0}$ is the first return time of $f_{0}$ to $X_0=(1/2,1]$. In this section we show that  $F$ satisfies (H1--H5)  for some appropriate function spaces $\cB, \cB_w$ described below.

In what follows we use the notation introduced in Section~\ref{sec-setup} for the study of the Markov example~\eqref{eq-2DLSV}  keeping in mind the new definition of $f_0, f, F_0, F$. Note that, for functions that depend only on $x$, the Banach space in the previous section was essentially reducing to the space of Lipsichtz functions. Here instead it will reduce to $BV$. This is natural, since $BV$ is the standard Banach space on which to analyse the spectrum of the transfer operator of a piecewise expanding map.

\subsection{Banach spaces}\label{subsec-Bsp_nM}

Consider the set of test functions $\cD_q=\{\phi\in L^\infty([0,1]^2,\bC)\;|\;\|\phi(x,\cdot)\|_{\cC^q}\leq 1, \; \textrm{ for almost all $x$}\}$ and
$\cD^0_q=\cD_q\cap \textrm{Lip}$.\footnote{ By  $\textrm{Lip}$ we mean the set of Lipsichtz functions.} With this, given $q\in (\frac{1+\beta}{2+\beta},1]$, for all $h\in BV$ we define the norms
\[
\begin{split}
&\|h\|_{\cB_w}=\sup_{\phi\in \cD_{1+q}}\int_Y h\cdot \phi\\
&\|h\|_0=\sup_{\phi\in \cD_{q}}\int_Y h\cdot \phi\\
&\|h\|_1=\sup_{\phi\in \cD^0_{1+q}}\int_Y h\cdot \partial_x\phi\\
\end{split}
\]
and set $\|h\|_{\cB}=\|h\|_1+\|h\|_0$. 
\begin{lemma}\label{lem:zero-norm}
For each $h\in BV$ we have
\[
\|h\|_{\cB_w}\leq\|h\|_{\cB}\leq \|h\|_{BV}.
\]
\end{lemma}
\begin{proof}
The first follows from $\|h\|_{\cB_w}\leq\|h\|_0$, which is obvious since the sup is taken on a larger set of functions. To see the second note that, on the one hand, for each $\phi\in \cD_{q}$ 
\[
\left|\int_Yh\phi\right|\leq \|h\|_{L^1}\leq \|h\|_{BV}.
\]
While, on the other hand, for all $\phi\in \cD_{1+q}^0$, let $\Phi=(\phi,0)\in C^0([0,1]^2,\bC^2)$. Then, for each $h\in BV$, 
\[
\|h\|_1=\sup_{\phi\in\cD_{1+q}^0}\int h\partial_x\phi=\sup_{\phi\in\cD_{1+q}^0}\int h\operatorname{div}\Phi
\leq \sup_{\|\Psi\|_{C^0}\leq 1}\int h\operatorname{div}\Phi=\|h\|_{BV},
\]
where, in the last equation, we have used the definition of the $BV$ norm in any dimension \cite{EG}.
\end{proof}

We can then define the Banach spaces $\cB_w$, $\cB$ obtained, respectively, by closing $BV$ with respect to $\|\cdot\|_{\cB_w}$ and $\|\cdot\|_{\cB}$. Note that such a definition (together with Lemma \ref{lem:zero-norm}) implies $BV\subset\cB\subset\cB_w$.
In fact, the next Lemma gives a more stringent embedding property.
\begin{lemma}
The unit ball of $\cB$ is relatively compact in $\cB_w$.
\end{lemma}
\begin{proof}
For each $\phi\in\cD^0_{1+q}$ define $\Psi(x,y)=\int_{1/2}^x\phi(z,y)dz$. Next, for each $\ve>0$ define $a_k=\frac 12+k\ve$ and consider the piecewise linear function
\[
\theta(x,y)=\int_{1/2}^{a_k}\phi(z,y)dy+\frac{x-a_k}\ve\int_{a_k}^{a_{k+1}}\phi(z,y)dz.
\]
One can easily check that $\ve^{-1}(\Psi-\theta)\in\cD_{1+q}^0$. Thus, for each $h\in BV$ belonging to the unit ball of $\cB$, we have
\[
\int_Yh \phi=\int_Y h\partial_x\Psi=\int_Y h\partial_x(\Psi-\theta)+\int_Y h\partial_x\theta=\int_Y h\partial_x\theta+ O(\ve).
\]
Setting $\alpha_k(y)=\ve^{-1}\int_{a_k}^{a_{k+1}}\phi(z,y)dz$ we have $\|\alpha_k\|_{\cC^{1+q}}\leq 1$ and 
\[
\partial_x\theta(x,y)=\sum_k\Id_{[a_k,a_{k+1}]}(x)\alpha_k(y).
\]
We can set $b_j=\ve j$ and define
\[
\begin{split}
&\ell_{k}(y)=\alpha_k(b_j)+\frac{\alpha_k(b_{j+1})-\alpha_k(b_j)}{\ve}(y-b_j)\quad \quad \textrm{for all }y\in[b_j,b_{j+1}]\\
&\ell(y)=\sum_k\Id_{[a_k,a_{k+1}]}(x)\ell_k(y).
\end{split}
\]
Note that $\ve^{-1}[\partial_x\theta-\ell]\in \cD_q$, thus
\[
\int_Yh \phi=\int_Y h\ell+ O(\ve).
\]
To conclude note that, by construction, for each $\ve$ the functions $\ell$ belong to a uniformly bounded set in a finite dimensional space (hence are contained in a compact set). The Lemma follows by the same arguments used at the end of Lemma \ref{lemma-embed}.
\end{proof}

\subsection{ Lasota-Yorke type inequality}

In the remaining of the paper, $R$ stands for the transfer operator associated with $F$ defined by~\ref{eq:transfer-L1} (with the current definition of $F$).
With this specified we state

\begin{lemma} \label{lem:LY-disc} For each $h\in BV$ and $z\in \bD$, we have
\[
\begin{split}
&\|R(z) h\|_{\cB_w}\leq |z|\|h\|_{\cB_w}\\
&\|R(z)h\|_{\cB}\leq \max\{2\lambda^{-1}, \lambda^{-q}\}|z| \|h\|_{\cB}+ C|z|\|h\|_{\cB_w}.
\end{split}
\]
\end{lemma}
\begin{proof}
For each $\phi\in\cD_{1+q}$ and $h\in BV$ we have
\[
\int_Y R(z) h\cdot \phi=\int_{Y}h z^{\vf}\phi\circ F.
\]
Note that, for almost all $x\in (1/2,1]$, $\psi_x(\cdot)=z^{\vf_0(x)}\phi\circ F(x, \cdot)$ is a $C^{1+q}$ function by condition \eqref{eq:extra-cond}.  Moreover, 
$\|\psi_x\|_{\cC^{1+q}}\leq |z|$ by the stable contraction of $F$. Hence, the first inequality follows.

If we have  $\phi\in\cD_{q}$
\[
\int_Y R(z) h\cdot \phi=\int_{Y}h (z^{\vf}\phi\circ F-\theta_z)+\int_{Y}h \theta_z,
\]
where $\theta_z(x,y)=z^{\vf_0(x)}\phi\circ F(x,0)$. Note that, for almost all $x\in (1/2, 1]$, $\|z^{\vf(x,\cdot)}\phi\circ F(x,\cdot)-
\theta_z(x,\cdot)\|_{\cC^q}\leq \lambda^{-q}|z|$, while $\|\theta_z\|_{\cC^{1+q}}\leq |z|$. Hence, $\lambda^q|z|^{-1}[z^{\vf}\phi\circ F-\theta_z]\in \cD_{q}$ and $|z|^{-1}\theta_z\in \cD_{1+q}$. Accordingly,
\[
\left|\int_Y R(z) h\cdot \phi\right|\leq\lambda^{-q}|z|\|h \|_0+C|z|\|h\|_{\cB_w}.
\]
To conclude, let $\phi\in \cD_{1+q}^0$. Then
\[
\int_Y R(z) h\cdot \partial_x\phi=\sum_j\int_{Y_j}h z^{\vf}(\partial_x\phi)\circ F.
\] 

Since $F\in C^2$ in each $Y_j$, we compute that
\begin{equation}\label{eq:deriv-ly}
\partial_x\frac{\phi\circ F}{F_0'}=\left[\partial_x\phi)\right]\circ F+\frac{\partial_y\phi\circ F\cdot\partial_x G}{F_0'}+\phi\circ F\cdot\partial_x (F_0')^{-1}.
\end{equation}
First of all, notice that there exists $C>0$ such that $C^{-1}z^\vf\phi\circ F\partial_x (F_0')^{-1}\in\cD_{1+q}$. Next, let 
$\theta_z(x,y)=\frac{z^{\vf_0(x)}\partial_y\phi\circ F(x,0)\cdot\partial_x G(x,0)}{F_0'(x)}$. Notice that $\|\theta_z\|_{\cC^{1+q}}\leq |z|$ and
\[
|z|^{-1}\lambda^q\left[\frac{z^{\vf}\partial_y\phi\circ F\cdot\partial_x G}{F_0'}-\theta_z\right]\in \cD_q.
\]
Putting the above together,\footnote{ Note that $\vf$ is constant on each $Y_j$ and hence can be moved inside the derivative.}
\[
\left|\int_Y R(z) h\cdot \partial_x\phi\right|\leq\left|\sum_j\int_{\overline Y_j}h \partial_x\left[z^\vf\frac{\phi\circ F}{F_0'}\right]\right|+\lambda^{-q}|z|\|h\|_0+C|z|\|h\|_{\cB_w}.
\]
Since $\Psi:=z^\vf\frac{\phi\circ F}{F_0'}$ is discontinuous it does not belong to $\cD^0_{1+q}$.
To take care of such a problem we introduce appropriate counter terms. 

Remark that each $Y_j$ is the union of, at most finitely many, connect sets of the form $Y_{j,m}=[a_{j,m},b_{j,m}]\times [0,1]$ for some $a_{j,m},b_{j,m}\in [0,1]$. On each $Y_{j,m}$ we define the function
\[
\ell_{j,m}(x,y)=\Psi(a_{j,m},y)+\frac{\Psi(b_{j,m},y)-\Psi(a_{j,m},y)}{b_{j,m}-a_{j,m}}(x-a_{j,m}),
\]
and $\ell=\sum_{j,m}\Id_{Y_{j,m}}\ell_{j,m}$. Note that $I_p^1=\{x\in I_p\;:\; \vf_0(x)>1\}$ consist of one single interval. Thus $f_0(I_p^1)\subset [0,\frac 12]$ is the union of intervals whose image will eventually cover all $(\frac 12,1]$ apart, at most, for two intervals at the boundary  of $f_0(I_p^1)$. By the usual distortion estimates this implies that there exists a constant $C>0$ such that, for all but finitely many of the above mentioned intervals $[a_{j,m},b_{j,m}]$, have  
\[
|b_{j,m}-a_{j,m}|\geq C|(F_0^j)'(a_{j,m})|^{-1}.
\]
But then the same estimates, possibly with a smaller $C$, for all intervals.

The above considerations together with condition \eqref{eq:extra-cond} imply that $\frac{\lambda}{2|z|}(\Psi-\ell)\in\cD^0_{1+q}$ while $C^{-1}\partial_x\ell\in\cD_{1+q}$. We can then conclude
\[
\left|\int_Y R(z) h\cdot \partial_x\phi\right|\leq \max\{2\lambda^{-1},\lambda^{-q}\}|z|\|h\|_{\cB}+C|z|\|h\|_{\cB_w}.
\]
\end{proof}

\subsection{ Checking Hypoteses (H1)-(H5)}
In this section we check the hypotheses needed to apply the abstract theory. As many arguments are similar to the ones in Section \ref{sec-setup} we will go over them very quickly.

By Lemma \ref{lem:zero-norm} $C^1\subset BV\subset \cB\subset \cB_w$. Moreover, if $h\in C^1$
\[
\left|\int h\phi\right|\leq \|h\|_{\cB_w}\|\phi\|_{C^{1+q}}.
\]
Hence,   $\cB_w\subset (C^{1+q})'$ and (H1)(i) is satisfied with $\alpha=1$ and $\gamma=1+q$. 

Next, let us discuss (H2). In fact, for later convenience, we will prove a slightly stronger result. Note that the connect components of $\vf^{-1}(Y)$ have the form $E=(a,b)\times [0,1]$ for some $(a,b)\subset (1/2,1]$.  Hence, for all $\phi\in\cD^0_{1+q}$,
\begin{equation}\label{eq:h2-nm}
\begin{split}
\int_E h\phi &=\int_Yh\partial_x \int_0^x\phi \Id_E\leq \|h\|_{\cB} \left\|\int_0^1 \phi(t,\cdot)\Id_E (t,\cdot)dt\right\|_{C^{1+q}}\\
&\leq\|h\|_{\cB}\,|b-a|=\|h\|_{\cB}\, m(E).
\end{split}
\end{equation}
To conclude we use the relation between $\mu$ and $m$ which is the same as in the Markov case.

Hypothesis (H3) does not depend on the Banach space; it is rather an assumption on the map, and is proven as in Section \ref{sec-setup}.

Next, we look at the hypotheses involving the transfer operator. Note that, for $h\in BV$, $R(z)h$ might fail to be in $BV$ due to possible unbounded oscillations in the vertical direction. Thus, even though the $\cB$ norm of $R(z)h$ is bounded by Lemma \ref{lem:LY-disc}, the function $R(z)h$ might fail to belong to $\cB$, since the latter is defined as the objects that are approximated by $BV$ function. Note however that $R_nh\in BV$ for each $n\in\bN$.\footnote{ This follows since $\Id_{Y_n}$ is a multiplier in $BV$ and, for each smooth function $T:Y_n\to Y$, $h\in \BV$ implies $h\circ T\in BV$.}  Since
\[
R(z)h=\sum_n z^n R_nh
\]
it follows that (H4)(i) implies $R(z)(BV)\subset \cB$, and hence (H1)(ii). 

We proceed thus to prove the, stronger, (H4)(iii).
Let $\phi\in\cD_q$, then
\[
\int R_nh \phi=\int h(\phi\circ F\Id_{Y_{n}}-\theta)+\int h\theta,
\]
where $\theta(x,y)=\phi\circ F(x,0)\Id_{Y_{n}}(x,0)$. Note that, for almost all $x$, by \eqref{eq:extra-cond1},
\[
\|\phi\circ F(x,\cdot)\Id_{Y_{n}}(x,\cdot)-\theta(x,\cdot)\|_{C^q}\leq \|\Id_{Y_n}\partial_y G\|_\infty^q\leq C m(Y_n),
\]
where we have used the limitation on the possible values of $q$.\footnote{ In fact, here is the only place were such a condition is used.}
Thus, since $Y_{n}=[a_{n},b_{n}]\times [0,1]$, arguing as in \eqref{eq:h2-nm},
\[
\left|\int R_nh \phi\right|\leq \|h\|_0Cm(Y_n)+\int h\partial_x\int_{0}^x\theta\leq C\|h\|_{\cB}m(Y_{n}).
\]
This takes care of $\|R_n h\|_0$. Next, let $\phi\in\cD^0_{1+q}$. Arguing like in the proof of Lemma \ref{lem:LY-disc} we obtain
\[
\begin{split}
\left|\int R_nh \phi\right|&= \left|\sum_m\int_{Y_{n,m}} h(\partial_x\phi)\circ F\right|\\
&=\left|\sum_m\int_{Y_{n,m}} h\left[\partial_x\left(\frac{\phi\circ F}{F_0'}\right)-\frac{\partial_y\phi\circ F\cdot\partial_x G}{F_0'}-\phi\circ F\partial_x( F_0')^{-1}\right]\right|.
\end{split}
\]
Then, again as in Lemma \ref{lem:LY-disc}, we introduce $\theta_m$ linear in $x$, so that the functions $\Id_{Y_{n,m}}\left[\frac{\phi\circ F}{F_0'}-\theta_m\right]$ are continuous. Remembering \eqref{eq:extra-cond1} and \eqref{eq:extra-cond} we readily obtain
\[
\left|\int R_nh \cdot \partial_x\phi\right|\leq Cm(Y_n)\|h\|_{0}+ Cm(Y_n)\|h\|_1
\]
from which the hypothesis follows.

The proof of (H1)(iii) and  (H5)(ii) goes more or less as in the Markov case (with trivial changes due to the different norm) once one remembers the topological mixing assumption of $f$. Lemma \ref{lem:LY-disc} proves (H5)(i). 

\paragraph{Acknowledgements.}
The research of both authors was partially supported by  MALADY Grant, ERC AdG 246953. We wish to thank an anonymous referee for very useful comments and suggestions.

\end{document}